\documentclass[10pt, reqno]{amsart}
\usepackage{graphicx, amssymb, amsmath, amsthm, color}
\numberwithin{equation}{section}

\newcommand{\R}{\mathbb{R}}
\newcommand{\C}{\mathbb{C}}
\newcommand{\wh}[1]{\widehat{#1}}

\renewcommand{\Re}{\text{Re}}
\renewcommand{\Im}{\text{Im}}
\newcommand{\norm}[1]{\| #1\|}

\newcommand{\eps}{\varepsilon}

\def\smallint{\begingroup\textstyle \int\endgroup}

\newcommand{\xonorm}[3]{\|#1\|_{L_t^{#2}L_x^{#3}}}
\newcommand{\xnorm}[4]{\|#1\|_{L_t^{#2}L_x^{#3}(#4\times\R^3)}}
\newcommand{\xonorms}[2]{\|#1\|_{L_{t,x}^{#2}}}

\newcommand{\xnorms}[3]{\|#1\|_{L_{t,x}^{#2}(#3\times\R^3)}}

\newcommand{\M}{\mathcal{M}}

\newcommand{\nsc}{\vert\nabla\vert^{s_c}}

\newtheorem{theorem}{Theorem}[section]

\newtheorem{lemma}[theorem]{Lemma}

\newtheorem{corollary}[theorem]{Corollary}

\newtheorem{proposition}[theorem]{Proposition}

\theoremstyle{definition}
\newtheorem{definition}[theorem]{Definition}
\newtheorem{remark}[theorem]{Remark}

\theoremstyle{remark}

\begin{document}

\title[3D Radial NLS]{The radial defocusing nonlinear Schr\"odinger equation in three space dimensions}
\author{Jason Murphy}
\address{Department of Mathematics, UCLA,
               Los Angeles, CA 90095-1555, USA}
        \email{murphy@math.ucla.edu}

\begin{abstract} We study the defocusing nonlinear Schr\"odinger equation in three space dimensions. We prove that any radial solution that remains bounded in the critical Sobolev space must be global and scatter. In the energy-supercritical setting, we employ a space-localized Lin--Strauss Morawetz inequality of Bourgain. In the inter-critical regime, we prove long-time Strichartz estimates and frequency-localized Lin--Strauss Morawetz inequalities. 
\end{abstract}
\maketitle


\section{Introduction}\label{section:introduction}
	We consider the initial-value problem for the defocusing nonlinear Schr\"odinger equation (NLS) on $\R_t\times\R_x^3$:
	\begin{equation}\label{nls}
	\left\{\begin{array}{ll}
	(i\partial_t+\Delta)u=\vert u\vert^p u,
	\\ u(0,x)=u_0(x).
	\end{array}\right.
	\end{equation}
	
	This equation enjoys a scaling symmetry, namely
		\begin{equation}\label{eq:scaling}
		u(t,x)\mapsto\lambda^{2/p}u(\lambda^2 t,\lambda x),
		\end{equation}
which defines a notion of criticality for \eqref{nls}. In particular, the only homogeneous $L_x^2$-based Sobolev space that is left invariant by this rescaling is $\dot{H}_x^{s_c}(\R^3)$, where the \emph{critical regularity} is defined by $s_c:=3/2-2/p$. We study \eqref{nls} in both the \emph{inter-critical} ($0<s_c<1$) and \emph{energy-supercritical} ($s_c>1$) settings. We show that any radial solution to \eqref{nls} that remains bounded in $\dot{H}_x^{s_c}$ must be global and scatter. 

We begin with a few definitions.

\begin{definition}[Solution] A function $u:I\times\R^3\to\C$ is a \emph{(strong) solution} to \eqref{nls} if it belongs to $C_t\dot{H}_x^{s_c}(K\times\R^3)\cap L_{t,x}^{5p/2}(K\times\R^3)$ for any compact $K\subset I$ and obeys the Duhamel formula 
	$$u(t)=e^{it\Delta}u_0-i\smallint_0^t e^{i(t-s)\Delta}(\vert u\vert^p u)(s)\,ds$$
for any $t\in I$. We call $I$ the \emph{lifespan} of $u$; we say that $u$ is a \emph{maximal-lifespan solution} if it cannot be extended to any strictly larger interval. If $I=\R$ we call $u$ \emph{global}.
\end{definition}

\begin{definition}[Scattering size, blowup]\label{def:blowup} We define the \emph{scattering size} of a solution $u:I\times\R^3\to\C$ to \eqref{nls} by
		$$S_I(u):=\iint_{I\times\R^3}\vert u(t,x)\vert^{5p/2}\,dx\,dt.$$

If there exists $t_0\in I$ so that $S_{[t_0,\sup I)}(u)=\infty$, then we say $u$ \emph{blows up (forward in time)}. If there exists $t_0\in I$ so that $S_{(\inf I,t_0]}(u)=\infty$, then we say $u$ \emph{blows up (backward in time)}.

If $u$ is a global solution to \eqref{nls} that obeys $S_\R(u)<\infty$, then standard arguments show that $u$ \emph{scatters}, that is, there exist $u_\pm\in\dot{H}_x^{s_c}(\R^3)$ so that
		$$\lim_{t\to\pm\infty}\norm{u(t)-e^{it\Delta}u_\pm}_{\dot{H}_x^{s_c}(\R^3)}=0.$$
\end{definition}
	
	Our main result is the following.
	
\begin{theorem}\label{thm:main} Let $s_c\in(0,\tfrac12)\cup(\tfrac34,1)\cup(1,\tfrac32).$
Suppose $u:I\times\R^3\to\C$ is a radial maximal-lifespan solution to \eqref{nls} such that $\norm{u}_{L_t^\infty\dot{H}_x^{s_c}(I\times\R^3)}<\infty$. Then $u$ is global and scatters, with
	$$S_\R(u)\leq C(\norm{u}_{L_t^\infty\dot{H}_x^{s_c}})$$
for some function $C:[0,\infty)\to[0,\infty).$ 
\end{theorem}

Equivalently, we see that any radial solution that fails to scatter must blow up its $\dot{H}_x^{s_c}$-norm. We note here that the cases covered in Theorem~\ref{thm:main} correspond to choosing $p\in(4/3,2)\cup(8/3,4)\cup(4,\infty)$. The range $p\in (2,8/3]$ was already addressed in the non-radial setting \cite{Mur:swamprat}. 

The motivation for Theorem~\ref{thm:main} comes from the celebrated global well-posedness and scattering results for the defocusing mass- and energy-critical NLS. In space dimension $d$ the equation \eqref{nls} is \emph{mass-critical} for $p=4/d$ and \emph{energy-critical} for $p=4/(d-2)$ (and $d\geq 3$). These equations are distinguished by the presence of a conserved quantity at the level of the critical regularity. For the mass-critical NLS, the critical regularity is $s_c=0$ and the conserved quantity is the \emph{mass}, defined by
			\begin{equation}\nonumber
			M[u(t)]:=\int_{\R^d}\vert u(t,x)\vert^2\,dx.
			\end{equation}
For the energy-critical NLS, the critical regularity is $s_c=1$ and the conserved quantity is the \emph{energy}, defined by
			\begin{equation}\nonumber
			E[u(t)]:=\int_{\R^d}\tfrac12\vert\nabla u(t,x)\vert^2+\tfrac{1}{p+2}\vert u(t,x)\vert^{p+2}\,dx.
			\end{equation}

For the defocusing mass- and energy-critical NLS, it is now known that arbitrary initial data in the critical Sobolev space lead to global solutions that scatter \cite{Bou, CKSTT, Dod3, Dod2, Dod1, Gri, KTV, KVZ, RV, Tao, TVZ:sloth, Vis, Vis2}. (See also \cite{DodF, KM:focusing, KTV,  KV:focusing, KVZ} for results in the focusing case.) One of the key difficulties in establishing these results was the fact that none of the known monotonicity formulae (that is, Morawetz estimates) scale like the mass or energy. It was Bourgain's \emph{induction on energy} technique that showed how one can move beyond this difficulty: by finding solutions that concentrate on a characteristic length scale, one can `break' the scaling symmetry of the problem and hence bring the available Morawetz estimates back into play, despite their non-critical scaling.

In the mass- and energy-critical settings one has some \emph{a priori} control over solutions due to the conservation laws. Specifically, in the defocusing setting one has uniform (in time) $\dot{H}_x^{s_c}$-bounds. For $s_c\notin\{0,1\}$, one has no such \emph{a priori} control. However, the success of the techniques developed to treat the mass- and energy-critical problems inspires the following conjecture: for any $s_c>0$, any solution to defocusing NLS that remains bounded in $\dot{H}_x^{s_c}$ will be global and scatter. 

Progress has been made on this conjecture in both the inter-critical and energy-supercritical settings in dimensions $d\geq 3$ \cite{KM:intercritical, KV:supercritical, MiaMurZhe, Mur:swamprat, Mur:half}. In three dimensions, the case $s_c=1/2$ was treated by Kenig--Merle \cite{KM:intercritical}, while the range $1/2<s_c\leq 3/4$ was addressed by the author in \cite{Mur:swamprat}. Restricting to the case of radial solutions, we consider all remaining cases for which $s_c>0$ in three dimensions. In particular, we address the energy-supercritical regime for the first time in three dimensions, as well as the range $0<s_c<1/2$ for the first time in any dimension. The origin of the restriction to the radial setting and the possibility of extending Theorem~\ref{thm:main} to the non-radial case will be discussed below.


\subsection{Outline of the proof of Theorem~\ref{thm:main}} We begin by recording a local well-posedness result for \eqref{nls}.

\begin{theorem}[Local well-posedness]\label{thm:lwp} Let $s_c\geq 0$. For any $u_0\in\dot{H}_x^{s_c}(\R^3)$ and $t_0\in\R$, there exists a unique maximal-lifespan solution $u:I\times\R^3\to\C$ to \eqref{nls} with $u(t_0)=u_0.$ Furthermore
	\begin{itemize}
	\item (Local existence) $I$ is an open neighborhood of $t_0$.
	\item (Blowup criterion) If $\sup I$ is finite, then $u$ blows up forward in time in the sense of Definition~\ref{def:blowup}. If $\inf I$ is finite, then $u$ blows up backward in time.
	\item (Scattering and wave operators) If $\sup I=\infty$ and $u$ does not blow up forward in time, then $u$ scatters forward in time. That is, there exists $u_+\in\dot{H}_x^{s_c}(\R^3)$ so that
		\begin{equation}\label{eq:scatter}
		\lim_{t\to\infty}\norm{u(t)-e^{it\Delta}u_+}_{\dot{H}_x^{s_c}(\R^3)}=0.
		\end{equation}
Conversely, for any $u_+\in\dot{H}_x^{s_c}(\R^3)$, there exists a unique solution to \eqref{nls} defined in a neighborhood of $t=\infty$ such that \eqref{eq:scatter} holds. The analogous statements hold backward in time.
	\item (Small data scattering) If $\norm{u_0}_{\dot{H}_x^{s_c}(\R^3)}$ is sufficiently small, then $u$ is global and scatters, with $S_\R(u)\lesssim\norm{u_0}_{\dot{H}_x^{s_c}(\R^3)}^{5p/2}.$ 
	\end{itemize}
\end{theorem}

This theorem follows from well-known arguments. Specifically, arguments from \cite{CazWei} suffice to establish the result first for data in the inhomogeneous space $H_x^{s_c}(\R^3)$ (see \cite{KV:clay, KV:supercritical}, for example). To remove the assumption that $u_0\in L_x^2(\R^3)$, one can use a stability result such as the following.

\begin{theorem}[Stability]\label{thm:stability} Let $s_c\geq 0$. Let $I$ be a compact time interval and $v:I\times\R^3\to\C$ a solution to $(i\partial_t+\Delta)v=\vert v\vert^p v+e$ for some function $e$. Assume that
		$$\norm{v}_{L_t^\infty\dot{H}_x^{s_c}(I\times\R^3)}\leq E\quad\text{and}\quad S_I(v)\leq L.$$
Let $t_0\in I$ and $u_0\in\dot{H}_x^{s_c}(\R^3)$. Then there exists $\eps_0=\eps_0(E,L)$ such that if
		$$\norm{u_0-v(t_0)}_{\dot{H}_x^{s_c}(\R^3)}+\xnorms{\vert\nabla\vert^{s_c}e}{10/7}{I}\leq\eps$$
for some $0<\eps<\eps_0$, then there exists a solution $u:I\times\R^3\to\C$ to \eqref{nls} with $u(t_0)=u_0$ and $S_I(u-v)\lesssim_{E,L}\eps.$ 
\end{theorem}

	This result also follows from well-known arguments, similar to those used to prove local well-posedness. For small-power nonlinearities (that is, $\vert u\vert^p u$ for $p<1$), the stability theory can become quite delicate; see \cite[Section~3.4]{KV:clay} for a detailed discussion with references. In our setting, however, the stability theory is fairly straightforward; one can find the necessary arguments in \cite{KV:supercritical, Mur:swamprat}. 

	We can now outline the proof of Theorem~\ref{thm:main}. We argue by contradiction and suppose that Theorem~\ref{thm:main} were false. As scattering holds for sufficiently small initial data (cf. Theorem~\ref{thm:lwp}), we deduce the existence of a threshold below which solutions scatter, but above which we can find solutions with arbitrarily large scattering size. Using a limiting argument, we can then prove the existence of nonscattering solutions living exactly at the threshold, that is, \emph{minimal counterexamples}. To complete the proof, we need only to rule out the existence of such minimal counterexamples. 
	
	To accomplish this, we make use of an important property of minimal counterexamples, namely \emph{almost periodicity} modulo the symmetries of the equation. Let us now briefly discuss this property and some of its consequences. For a more extensive treatment, one can refer to \cite{KV:clay}. 

\begin{definition}[Almost periodic solutions] Let $s_c>0$. A solution $u:I\times\R^3\to\C$ to \eqref{nls} is \emph{almost periodic (modulo symmetries)} if 
	\begin{equation}\label{a priori}
	u\in L_t^\infty\dot{H}_x^{s_c}(I\times\R^3)
	\end{equation} 
and there exist functions $x:I\to\R^3$, $N:I\to\R^+$, and $C:\R^+\to\R^+$ so that for all $t\in I$ and $\eta>0$, 
	$$\int_{\vert x-x(t)\vert>\frac{C(\eta)}{N(t)}} \big\vert \vert\nabla\vert^{s_c}u(t,x)\big\vert^2\,dx
		+\int_{\vert \xi\vert> C(\eta)N(t)}\vert\xi\vert^{2s_c}\vert\wh{u}(t,x)\vert^2\,d\xi<\eta.$$
We call $N(t)$ the \emph{frequency scale function}, $x(t)$ the \emph{spatial center}, and $C(\eta)$ the \emph{compactness modulus function}.
\end{definition}

\begin{remark}\label{remark:radial} As a consequence of radiality, the solutions we consider can only concentrate near the spatial origin. In particular, we may take $x(t)\equiv 0$. 
\end{remark}

The frequency scale function of an almost periodic solution obeys the following local constancy property (see \cite[Lemma 5.18]{KV:clay}).

\begin{lemma}[Local constancy]\label{lem:constancy}  If $u:I\times\R^3\to\C$ is a maximal-lifespan almost periodic solution to \eqref{nls}, then there exists $\delta=\delta(u)>0$ so that for all $t_0\in I$
		$$[t_0-\delta N(t_0)^{-2},t_0+\delta N(t_0)^{-2}]\subset I.$$
Moreover, $N(t)\sim_u N(t_0)$ for $\vert t-t_0\vert\leq \delta N(t_0)^{-2}.$
\end{lemma}

Thus, for an almost periodic solution $u:I\times\R^3\to\C$, we may divide the lifespan $I$ into \emph{characteristic subintervals} $J_k$ on which we can set $N(t)\equiv N_k$ for some $N_k$, with $\vert J_k\vert \sim_u N_k^{-2} $. This requires us to modify the compactness modulus function by a time-independent multiplicative factor.

Lemma~\ref{lem:constancy} also provides information about the behavior of the frequency scale function at the blowup time. In particular, we have the following (see \cite[Corollary~5.19]{KV:clay}).

\begin{corollary}[$N(t)$ at blowup]\label{cor:constancy} Let $u:I\times\R^3\to\C$ be a maximal-lifespan almost periodic solution to \eqref{nls}. If $T$ is a finite endpoint of $I$, then $N(t)\gtrsim_u \vert T-t\vert^{-1/2}$. 
\end{corollary}

The next lemma relates the frequency scale function of an almost periodic solution to its Strichartz norms.

\begin{lemma}[Spacetime bounds]\label{lem:sb} Let $u:I\times\R^3\to\C$ be an almost periodic solution to \eqref{nls}. Then
	$$\smallint_I N(t)^2\,dt\lesssim_u \xnorm{\vert\nabla\vert^{s_c}u}{2}{6}{I}^2\lesssim_u 1+\smallint_I N(t)^2\,dt.$$
\end{lemma}

To prove this lemma, one can adapt the proof of \cite[Lemma~5.21]{KV:clay}. The key is to notice that $\int_I N(t)^2\,dt$ counts the number of characteristic subintervals $J_k$ inside $I$, and that for each such subinterval we have
	$$\xnorm{\vert\nabla\vert^{s_c}u}{2}{6}{J_k}\sim_u 1.$$

We record here some final consequences of almost periodicity (see the proofs of (7.2) and (7.4) in \cite{MiaMurZhe} for very similar arguments).

\begin{lemma}\label{lower bounds} Let $u:I\times\R^3\to\C$ be a nonzero almost periodic solution to \eqref{nls} such that $x(t)\equiv 0$. Then there exists $C(u)>0$ so that
	\begin{equation}\label{lb1}
	\inf_{t\in I}\big(N(t)^{2s_c}\smallint_{\vert x\vert\leq\frac{C(u)}{N(t)}}\vert u(t,x)\vert^2\,dx\big)\gtrsim_u 1.
	\end{equation}
If $\inf_{t\in I}N(t)\geq 1$, then there exists $C(u)>0$ and $N_0>0$ so that for $N<N_0$,
		\begin{equation}\label{lb3}
	\inf_{t\in I}\big(N(t)^{2s_c}\smallint_{\vert x\vert\leq\frac{C(u)}{N(t)}}\vert u_{>N}(t,x)\vert^2\,dx\big)\gtrsim_u1.
	\end{equation}	
If $\sup_{t\in I}N(t)\leq 1$, then there exists $C(u)>0$ and $N_1>0$ so that for $N>N_1$,	
		\begin{equation}\label{lb2}
	\inf_{t\in I}\big(N(t)^{2s_c}\smallint_{\vert x\vert\leq\frac{C(u)}{N(t)}}\vert u_{\leq N}(t,x)\vert^2\,dx\big)\gtrsim_u 1.
	\end{equation}
\end{lemma}

We can now state the first major step in the proof of Theorem~\ref{thm:main}.

\begin{theorem}[Reduction to almost periodic solutions]\label{thm:red1} If Theorem~\ref{thm:main} fails, then there exists a radial maximal-lifespan solution $u:I\times\R^3\to\C$ to \eqref{nls} that is almost periodic and blows up forward and backward in time.
\end{theorem}

The reduction to almost periodic solutions is now fairly standard in the field of nonlinear dispersive PDE, especially in the setting of NLS. Keraani \cite{Ker} originally established the existence of minimal blowup solutions to NLS, while Kenig--Merle \cite{KM:focusing} were the first to use them as a tool to prove global well-posedness. This technique has since been used in a variety of settings and has proven to be extremely effective. One can refer to \cite{HR, KM:intercritical, KTV, KV:focusing, KV:supercritical, KV:revisit, KVZ, MiaMurZhe, Mur:swamprat, Mur:half, TVZ, Vis:revisit} for some examples in the case of NLS; see also \cite{KV:clay} for a good introduction to the method.

The proof of Theorem~\ref{thm:red1} requires three main ingredients. The first is a linear profile decomposition for the free propagator $e^{it\Delta}$; the particular result we need can be found in \cite{Sha}. The second ingredient is a stability result, such as Theorem~\ref{thm:stability} above. Finally, one needs to prove a decoupling statement for nonlinear profiles. Arguments of Keraani \cite{Ker:Strichartz} suffice to establish such a decoupling in the mass- and energy-critical settings; however, these arguments rely on pointwise estimates and hence fail to be directly applicable in the presence of fractional derivatives (for example, when $s_c\notin\{0,1\}$). 

Some strategies are now available for dealing with this technical issue. In \cite{KV:supercritical}, Killip--Visan devised a strategy that works in the energy-supercritical regime, while some alternate approaches are available in the inter-critical setting \cite{HR, KM:intercritical, Mur:swamprat, Mur:half}. As for the cases we consider, the arguments of \cite{KV:supercritical} apply in the cases for which $s_c>1$, while the arguments of \cite{Mur:swamprat} apply in the cases for which $0<s_c<1$.

Next, we make some refinements to the solutions given in Theorem~\ref{thm:red1}. 

If $s_c>1/2$, we proceed as follows. First, a rescaling argument as in \cite{KTV, KV:focusing, TVZ:sloth} allows us to restrict attention to almost periodic solutions as in Theorem~\ref{thm:red1} that do not escape to arbitrarily low frequencies on at least half of their lifespan, say $[0,T_{max})$. Next, as described above, we divide $[0,T_{max})$ into characteristic subintervals $J_k$ so that $N(t)\equiv N_k$ on $J_k$, with $\vert J_k\vert\sim_u N_k^{-2}$. Finally, we recall from Remark~\ref{remark:radial} that we may take $x(t)\equiv 0$. Altogether, we get the following.

\begin{theorem}\label{thm:reduction1} Fix $s_c\in(\tfrac34,1)\cup(1,\tfrac32)$ and suppose Theorem~\ref{thm:main} fails. Then there exists a radial almost periodic solution $u:[0,T_{max})\times\R^3\to\C$ such that $S_{[0,T_{max})}=\infty$ and
	\begin{equation}\nonumber
	N(t)\equiv N_k\geq 1\quad \text{for }t\in J_k,\ 
	\text{with}\ [0,T_{max})=\cup J_k\ \text{and}\ \vert J_k\vert\sim_u N_k^{-2}. 
	\end{equation}
\end{theorem}

If $s_c<1/2$, we proceed similarly; however, we instead restrict attention to almost periodic solutions as in Theorem~\ref{thm:red1} that do not escape to arbitrarily \emph{high} frequencies on $[0,T_{max})$. In light of Corollary~\ref{cor:constancy}, this implies $T_{max}=\infty$. We arrive at the following.

\begin{theorem}\label{thm:reduction2} Fix $s_c\in(0,\tfrac12)$ and suppose Theorem~\ref{thm:main} fails. Then there exists a radial almost periodic solution $u:[0,\infty)\times\R^3\to\C$ such that $S_{[0,\infty)}(u)=\infty$ and
	\begin{equation}
	N(t)\equiv N_k\leq 1\quad\text{for }t\in J_k,\ \nonumber	
	\text{with }[0,\infty)=\cup J_k\ \text{and}\ \vert J_k\vert\sim_u N_k^{-2}.
	\end{equation}
\end{theorem}

To prove Theorem~\ref{thm:main}, it therefore suffices to rule out the existence of almost periodic solutions as in Theorem~\ref{thm:reduction1} and Theorem~\ref{thm:reduction2}. As alluded to above, the best tools available are monotonicity formulae that hold for defocusing NLS, which are known as \emph{Morawetz estimates}. In this paper, we will use versions of the \emph{Lin--Strauss Morawetz inequality} of \cite{LinStr}, which is given by
	\begin{equation}\label{eq:ls}
		\iint_{I\times\R^d}\frac{\vert u(t,x)\vert^{p+2}}{\vert x\vert}\,dx\,dt\lesssim
		\norm{\vert\nabla\vert^{1/2}u}_{L_t^\infty L_x^2(I\times\R^d)}^2.
		\end{equation}

Because of the weight $\tfrac{1}{\vert x\vert}$, the Lin--Strauss Morawetz inequality is well suited for preventing concentration near the origin, and hence it is most effective in the radial setting. In fact, it is the use of Lin--Strauss Morawetz inequality that leads to the restriction to the radial setting in Theorem~\ref{thm:main}. We cannot use this estimate directly, however, as the solutions we consider need only belong to $L_t^\infty\dot{H}_x^{s_c}$ (and hence the right-hand side of \eqref{eq:ls} need not be finite). 

For $s_c>1/2$, one needs to suppress the low frequencies of solutions in order to access the estimate \eqref{eq:ls}. In his work on the 3D radial energy-critical NLS \cite{Bou}, Bourgain accomplished this by proving a space-localized version of \eqref{eq:ls} (see also \cite{Gri, Tao}). In Section~\ref{section:case one}, we adopt the same approach to treat the range $1<s_c<3/2$. As one of the error terms resulting from space-localization requires control of the solution at the level of $\dot{H}_x^1$, a different approach is needed to handle the range $3/4<s_c<1$. In particular, in Section~\ref{section:case two} we prove a version of \eqref{eq:ls} localized to high frequencies. For $s_c<1/2$, one instead needs to suppress the high frequencies of solutions in order to access \eqref{eq:ls}. In Section~\ref{section:case three} we prove a version of \eqref{eq:ls} localized to low frequencies. 

The remainder of the paper is organized as follows:

In Section~\ref{section:lemmas} we record some notation and background results.
 
 In Section~\ref{section:case one} we rule out almost periodic solutions as in Theorem~\ref{thm:reduction1} with $1<s_c<3/2$. We consider separately the cases $T_{max}<\infty$ and $T_{max}=\infty$. In Section~\ref{section:ftbu} we deal with the case $T_{max}<\infty$. We show that the existence of such solutions is inconsistent with the conservation of energy. We rely on the following `reduced' Duhamel formula for almost periodic solutions to \eqref{nls} (see \cite[Proposition 5.2]{KV:clay}).  

\begin{proposition}[Reduced Duhamel formula]\label{reduced duhamel} Let $u:I\times\R^3\to\C$ be a maximal-lifespan almost periodic solution to \eqref{nls}. Then for all $t\in I$, we have
	$$u(t)=\lim_{T\nearrow\sup I} i\smallint_t^T e^{i(t-s)\Delta}(\vert u\vert^p u)(s)\,ds$$
as a weak limit in $\dot{H}_x^{s_c}(\R^3).$  
\end{proposition}

In Section~\ref{section:mor} we treat the case $T_{max}=\infty$. As mentioned above, we use a space-localized Lin--Strauss Morawetz inequality as in the work of Bourgain \cite{Bou}. Combining this estimate with Lemma~\ref{lower bounds}, we derive the bound
	$$\smallint_I N(t)^{3-2s_c}\,dt\lesssim_u \vert I\vert^{s_c-1/2}$$
for any compact interval $I\subset[0,\infty)$. Recalling that $\inf N(t)\geq 1$ and $s_c<3/2$, we reach a contradiction by taking $I$ sufficiently large inside $[0,\infty)$. 

Sections~\ref{section:case two}~and~\ref{section:case three} are devoted to ruling out solutions as in Theorem~\ref{thm:reduction1} and Theorem~\ref{thm:reduction2} for which $s_c\in(0,\tfrac12)\cup(\tfrac34,1)$. In both regimes we break into two cases depending on the control given by the Lin--Strauss Morawetz inequality. Specifically, we consider separately the cases
	$$\smallint_0^{T_{max}} N(t)^{3-2s_c}\,dt<\infty\quad\text{and}\quad\smallint_0^{T_{max}} N(t)^{3-2s_c}\,dt=\infty,$$
where $T_{max}=\infty$ for $s_c\in(0,\tfrac12)$. We refer to the first case as the \emph{rapid frequency-cascade} scenario and to the second as the \emph{quasi-soliton} scenario. 

In Section~\ref{section:case two} we rule out solutions as in Theorem~\ref{thm:reduction1} with $3/4<s_c<1$. (In fact, the arguments are equally valid in the range $1/2<s_c\leq 3/4$; however, these cases have already been treated in the non-radial setting \cite{Mur:swamprat}.) The main technical tool we use is a \emph{long-time Strichartz estimate}, which appears as Proposition~\ref{prop:LTS2}. This type of estimate was originally developed by Dodson \cite{Dod3} for the mass-critical NLS; similar estimates have since appeared in the energy-critical, inter-critical, and energy-supercritical settings \cite{KV:revisit, MiaMurZhe, Mur:swamprat, Vis:revisit}. In all previous cases, the estimates have been adapted to the interaction Morawetz inequality. We instead prove long-time Strichartz estimates adapted to the Lin--Strauss Morawetz inequality.

In Section~\ref{section:rfc2} we rule out the rapid frequency-cascade scenario. We use Proposition~\ref{prop:LTS2} to show that such a solution would possess additional decay. We then use the conservation of mass to derive a contradiction.

In Section~\ref{section:flmor2} we establish a frequency-localized Lin--Strauss Morawetz inequality (Proposition~\ref{prop:flmor2}). We begin as in the proof of the standard Lin--Strauss Morawetz inequality, truncate the solution to high frequencies, and use Proposition~\ref{prop:LTS2} to control the resulting error terms.

In Section~\ref{section:soliton2} we use Proposition~\ref{prop:flmor2} to rule out the quasi-soliton scenario. In particular, we combine Proposition~\ref{prop:flmor2} with Lemma~\ref{lower bounds} to deduce a uniform bound on $\int_I N(t)^{3-2s_c}\,dt$ for any compact interval $I\subset[0,T_{max})$. We reach a contradiction by taking $I$ to be sufficiently large inside $[0,T_{max})$. 

In Section~\ref{section:case three} we rule out the existence of solutions as in Theorem~\ref{thm:reduction2}, in which case $0<s_c<1/2$. The structure of Section~\ref{section:case three} is quite similar to that of Section~\ref{section:case two}. In particular, we prove a long-time Strichartz estimate (Proposition~\ref{prop:LTS}) and a frequency-localized Lin--Strauss Morawetz inequality (Proposition~\ref{thm:flmor}). However, as mentioned above, in the range $0<s_c<1/2$ we need to work with the low frequency component of solutions rather than the high frequencies, and hence the details of the analysis change. We find that rapid frequency-cascades now possess additional regularity, which can be combined with the conservation of energy to derive a contradiction. For the quasi-soliton scenario, we again use a frequency-localized Morawetz inequality to derive a uniform bound on $\int_I N(t)^{3-2s_c}\,dt$, which yields the same contradiction as before.

To conclude, we discuss the possibility of extending Theorem~\ref{thm:main} to the non-radial setting. The most natural approach would be to work with the interaction Morawetz inequality (introduced in \cite{CKSTT0}) instead of the Lin--Strauss Morawetz inequality. In particular, for solutions $u:I\times\R^d\to\C$ to defocusing NLS, one has the following estimate:
	\begin{align}
	-\int_I\iint_{\R^d\times\R^d} &\vert u(t,x)\vert^2\Delta\big(\tfrac{1}{\vert x-y\vert}\big)\vert u(t,y)\vert^2\,dx\,dy\,dt
	\nonumber
	\\ &\lesssim \norm{\vert\nabla\vert^{1/2}u}_{L_t^\infty L_x^2(I\times\R^d)}^2\norm{u}_{L_t^\infty L_x^2(I\times\R^d)}^2.\label{eq:im}
	\end{align} 

The estimate \eqref{eq:im} controls the degree to which mass can interact with itself throughout all of $\R^d$, and hence it is useful even in the non-radial setting. As above, however, the right-hand side of \eqref{eq:im} may not be finite, and hence the estimate may not be directly applicable. In their pioneering work on the energy-critical NLS \cite{CKSTT}, Colliander--Keel--Staffilani--Takaoka--Tao addressed this issue by proving a frequency-localized version of \eqref{eq:im}. This approach has since been adopted in many different settings \cite{Dod3, Dod2, Dod1, DodF, KV:revisit, MiaMurZhe, Mur:swamprat, RV, Vis, Vis2, Vis:revisit}. To prove such an estimate, one must control the error terms that arise in the standard interaction Morawetz inequality when one applies a frequency cutoff to a solution to NLS. The long-time Strichartz estimates in the work of Dodson and others are designed to handle such error terms.

In \cite{Mur:swamprat}, the author proved long-time Strichartz estimates adapted to the interaction Morawetz inequality in three dimensions for $1/2<s_c<1$. In this range, one works with the high-frequency component of solutions to guarantee the finiteness of the right-hand side of \eqref{eq:im}. However, for $s_c>3/4$, one of the error terms that arises from the frequency cutoff cannot be controlled unless one also imposes a spatial truncation; this is the approach taken in the energy-critical setting, for example \cite{CKSTT, KV:revisit}. This spatial truncation results in additional error terms that require control over the solution at the level of $\dot{H}_x^1$. Thus, in the energy-critical case, one can use the conservation of energy and ultimately push the arguments through, while in the case $3/4<s_c<1$, one would need significant additional input to control the additional error terms. For a further discussion, see \cite{Mur:swamprat}.

For $s_c\geq 1$ in three dimensions, it becomes difficult to prove long-time Strichartz estimates adapted to the interaction Morawetz inequality. Killip--Visan  \cite{KV:revisit} were able to prove such an estimate at the level $s_c=1$, but they also show that their estimate is sharp in the sense that it is saturated by the ground state solution to the focusing equation. Thus in this case no stronger estimate could be proven without explicitly incorporating the defocusing nature of the equation. For $s_c>1$ there is no ground state to contend with, and hence it may still be possible to establish a long-time Strichartz estimate adapted to the interaction Morawetz in this setting; however, this seems to be a difficult problem and we have been unable to establish such an estimate. 

Finally, for $0<s_c<1/2$ there is also some difficulty associated to proving a frequency-localized interaction Morawetz inequality. In particular, as one needs to control both the $L_x^2$- and $\dot{H}_x^{1/2}$-norms of the solution in \eqref{eq:im}, one would need to truncate both the low and high frequencies of the solution. While this approach may ultimately prove to be tractable, we did not pursue this direction here. Instead, we chose to work with the Lin--Strauss Morawetz inequality, in which case it suffices to truncate the high frequencies only. As mentioned above, the use of the Lin--Strauss Morawetz inequality necessitates the restriction of our final result to the radial setting.

\subsection*{Acknowledgements} I am grateful to Rowan Killip and Monica Visan for helpful comments. This work was supported in part by NSF grant DMS-1265868 (P.I. Rowan Killip).


\section{Notation and useful lemmas}\label{section:lemmas}
\subsection{Some notation}  
For nonnegative $X,Y$ we write $X\lesssim Y$ to denote $X\leq CY$ for some $C>0$. If $X\lesssim Y\lesssim X$, we write $X\sim Y.$ The dependence of implicit constants on parameters will be indicated by subscripts, e.g. $X\lesssim_u Y$ denotes $X\leq CY$ for some $C=C(u)$. Dependence of constants on the ambient dimension or the power $p$ will be not be explicitly indicated.
		
	We use the expression $\text{\O}(X)$ to denote a finite linear combination of terms that resemble $X$ up to Littlewood--Paley projections, maximal functions, and complex conjugation. 
	
	We denote the nonlinearity $\vert u\vert^p u$ by $F(u)$. At various points throughout the paper, we will use the following basic pointwise estimates:
	\begin{align*}
	&\big\vert \vert u+v\vert^{p}(u+v)-\vert u\vert^{p}u\big\vert\lesssim \vert v\vert^{p+1}+\vert v\vert\,\vert u\vert^{p},
	\\ &\big\vert \vert u+v\vert^{p+2}-\vert u\vert^{p+2}-\vert v\vert^{p+2}\big\vert\lesssim \vert u\vert\,\vert v\vert^{p+1}+\vert u\vert^{p+1}\vert v\vert.
	\end{align*}

	For a time interval $I$, we write $L_t^qL_x^r(I\times\R^3)$ for the Banach space of functions $u:I\times\R^3\to\C$ equipped with the norm
		$$\xnorm{u}{q}{r}{I}:=\big(\smallint_I \norm{u(t)}_{L_x^r(\R^3)}^q\,dt\big)^{1/q},$$
with the usual adjustments when $q$ or $r$ is infinity. When $q=r$, we write $L_t^q L_x^q=L_{t,x}^q$. We will often abbreviate $\xnorm{f}{q}{r}{I}=\xonorm{f}{q}{r}$ and $\norm{f}_{L_x^r(\R^3)}=\norm{f}_{L_x^r}.$ 

	We define the Fourier transform on $\R^3$ by
		$$\wh{f}(\xi):=(2\pi)^{-3/2}\int_{\R^3}e^{-ix\cdot\xi}f(x)\,dx.$$
The fractional differentiation operator $\vert\nabla\vert^s$ is then defined by $\wh{\vert\nabla\vert^s f}(\xi):=\vert\xi\vert^s\wh{f}(\xi)$, and the corresponding homogeneous Sobolev norm is given by
		$$\norm{f}_{\dot{H}_x^s(\R^3)}:=\norm{\vert\nabla\vert^s f}_{L_x^2(\R^3)}.$$ 
		
\subsection{Basic harmonic analysis} We recall the standard Littlewood--Paley projection operators. Let $\varphi$ be a radial bump function supported on $\{\vert\xi\vert\leq11/10\}$ and equal to one on $\{\vert \xi\vert\leq 1\}.$ For $N\in 2^{\mathbb{Z}}$, we define 
	\begin{align*}
		&\wh{P_{\leq N} f}(\xi):=\wh{f_{\leq N}}(\xi):=\varphi(\xi/N)\wh{f}(\xi),
	\\	&\wh{P_{>N} f}(\xi) :=\wh{f_{>N}}(\xi):=(1-\varphi(\xi/N))\wh{f}(\xi),
	\\	&\wh{P_N f}(\xi) :=\wh{f_N}(\xi):=(\varphi(\xi/N)-\varphi(2\xi/N))\wh{f}(\xi).
	\end{align*}
We also define 
	$$P_{M<\cdot\leq N}:=P_{\leq N}-P_{\leq M}.$$ 

The Littlewood--Paley projection operators commute with derivative operators and the free propagator. They are self-adjoint, bounded on every $L_x^p$ and $\dot{H}_x^s$ space for $1\leq p\leq\infty$ and $s\geq 0$, and  bounded pointwise by the Hardy--Littlewood maximal function. They also obey the following standard Bernstein estimates.

\begin{lemma}[Bernstein] For $1\leq r\leq q\leq\infty$ and $s\geq 0$, 
	\begin{align*}
	\norm{\vert\nabla\vert^s f_N}_{L_x^r(\R^3)}&\sim N^s\norm{f_N}_{L_x^r(\R^3)},
	\\ \norm{\vert\nabla\vert^s f_{\leq N} }_{L_x^r(\R^3)}&\lesssim N^s\norm{f_{\leq N} }_{L_x^r(\R^3)},
	\\ \norm{f_{>N}}_{L_x^r(\R^3)}&\lesssim N^{-s}\norm{\vert\nabla\vert^s f_{>N}}_{L_x^r(\R^3)},
	\\ \norm{f_{\leq N} }_{L_x^q(\R^3)}&\lesssim N^{\frac{3}{r}-\frac{3}{q}}\norm{f_{\leq N}}_{L_x^r(\R^3)}.
	\end{align*}
\end{lemma}

We will use the following fractional calculus estimates from \cite{ChrWei}.

\begin{lemma}[Fractional product rule \cite{ChrWei}] Let $s> 0$ and $1<r,r_j,q_j<\infty$ satisfy $\tfrac{1}{r}=\tfrac{1}{r_j}+\tfrac{1}{q_j}$ for $j=1,2$. Then
	$$\norm{\vert\nabla\vert^s(fg)}_{L_x^r}\lesssim\norm{f}_{L_x^{r_1}}\norm{\vert\nabla\vert^s g}_{L_x^{q_1}}
							+\norm{\vert\nabla\vert^s f}_{L_x^{r_2}}\norm{g}_{L_x^{q_2}}.$$
\end{lemma}

\begin{lemma}[Fractional chain rule \cite{ChrWei}] Suppose $G\in C^1(\C)$ and $s\in(0,1]$ Let $1<r,r_2<\infty$ and $1<r_1\leq\infty$ be such that $\tfrac{1}{r}=\tfrac{1}{r_1}+\tfrac{1}{r_2}.$ Then
	$$\norm{\vert\nabla\vert^s G(u)}_{L_x^r}\lesssim\norm{G'(u)}_{L_x^{r_1}}\norm{\vert\nabla\vert^s u}_{L_x^{r_2}}.$$ 
\end{lemma}

We next record a paraproduct estimate in the spirit of \cite[Lemma~2.3]{Vis:revisit} (see also \cite{MiaMurZhe, Mur:swamprat, RV, Vis, Vis2}). We will need this estimate in Section~\ref{section:rfc2}. 

\begin{lemma}[Paraproduct estimate]\label{lem:paraproduct} Let $1<r<r_1<\infty$, $1<r_2<\infty$, and $0<s<1$ satisfy $\tfrac{1}{r_1}+\tfrac{1}{r_2}=\tfrac{1}{r}+\tfrac{s}{3}<1.$ Then
	$$\norm{\vert\nabla\vert^{-s}(fg)}_{L_x^r(\R^3)}\lesssim\norm{\vert\nabla\vert^{-s}f}_{L_x^{r_1}(\R^3)}
		\norm{\vert\nabla\vert^s g}_{L_x^{r_2}(\R^3)}$$
\end{lemma}

We also record Hardy's inequality, which appears in Sections~\ref{section:flmor2}~ and ~\ref{section:flmor}.

\begin{lemma}[Hardy] For $0< s<d$ and $1<r<d/s$, 
	$$\norm{\tfrac{1}{\vert x\vert^s}f}_{L_x^r(\R^d)}\lesssim\norm{\vert\nabla\vert^s f}_{L_x^r(\R^d)}.$$
\end{lemma}

Using Hardy's inequality and interpolation, we can also derive the following estimate for $0\leq s\leq 1$:
	\begin{equation}\label{hardys}	\norm{\vert\nabla\vert^{s}(\tfrac{x}{\vert x\vert} u)}_{L_x^2}\lesssim \norm{\vert\nabla\vert^s u}_{L_x^2}.
	\end{equation}

Finally, we record a few facts about the Hardy--Littlewood maximal function, which we denote by $\M$. Along with the standard maximal function estimate (i.e. the fact that $\M$ is bounded on $L_x^r$ for $1<r\leq\infty$), we will use the fact that
	$$\norm{\nabla\M (f)}_{L_x^r}\lesssim\norm{\nabla f}_{L_x^r}$$
for $1<r<\infty$ (see \cite{Kin}, for example).

\subsection{Strichartz estimates} We record here the standard Strichartz estimates for the Schr\"odinger equation, as well as a bilinear variant that will be used in Section~\ref{section:lts}.

We begin with a definition. 

\begin{definition}[Admissible pairs]\label{def:admissible} We call a pair of exponents $(q,r)$ \emph{admissible} if
	$$\tfrac{2}{q}+\tfrac{3}{r}=\tfrac{3}{2}\quad\text{and}\quad 2\leq q\leq\infty.$$
For a time interval $I$, we define
	$$\norm{u}_{S(I)}:=\sup\big\{\xnorm{u}{q}{r}{I}:(q,r)\ \text{admissible}\big\}.$$ 
We define $S(I)$ to be the closure of the test functions under this norm.
\end{definition}

We can now state the standard Strichartz estimates in the form that we need.

\begin{proposition}[Strichartz \cite{GinVel, KeeTao, Strichartz}]\label{lem:strichartz} Let $u:I\times\R^3\to\C$ be a solution to $$(i\partial_t+\Delta)u=F$$ and let $s\geq 0$. Then
	$$\norm{\vert\nabla\vert^s u}_{S(I)}\lesssim\norm{u(t_0)}_{\dot{H}_x^{s}}+\min\big\{\xnorm{\vert\nabla\vert^s F}{2}{6/5}{I},\xnorms{\vert\nabla\vert^s F}{10/7}{I}\big\}.$$
for any $t_0\in I$. 
\end{proposition}

We next record a bilinear Strichartz estimate. The particular version we need can be deduced from \cite[Corollary~4.19]{KV:clay}.

\begin{proposition}[Bilinear Strichartz]\label{lem:bilinear} Let $0<s_c<1$. For a time interval $I$ and frequencies $M,N>0$, we have
	$$\xnorms{u_{\leq M}v_{>N}}{2}{I}\lesssim M^{1-s_c}N^{-1/2}\norm{\vert\nabla\vert^{s_c} 
		u}_{S^*(I)}\norm{v_{>N}}_{S^*(I)},$$
where
	$$\norm{u}_{S^*(I)}:=\norm{u}_{L_t^\infty L_x^2(I\times\R^3)}+\xnorms{(i\partial_t+\Delta)u}{10/7}{I}.$$ 
\end{proposition}

We use Proposition~\ref{lem:bilinear} in Section~\ref{section:lts}. In that setting, we take $u=v$ an almost periodic solution to \eqref{nls} and $I=J_k$ a characteristic subinterval. In particular, we use the following corollary.

\begin{corollary}\label{cor:bilinear} Let $0<s_c<1$ and $s>0$. Suppose $u$ is an almost periodic solution to \eqref{nls} such that $u\in L_t^\infty\dot{H}_x^s$ and let $J_k$ be a characteristic subinterval. Then
	$$\xnorms{u_{\leq M}u_{>N}}{2}{J_k}\lesssim_u M^{1-s_c}N^{-1/2-s}.$$
\end{corollary}

\begin{proof} This result will follow from Proposition~\ref{lem:bilinear}, provided we can show
		\begin{equation}
		\label{bounds on Jk}
		\norm{\vert\nabla\vert^{s_c} u}_{S^*(J_k)}\lesssim_u 1
		\end{equation}
and
		\begin{equation}
		\label{gain on Jk}
		\norm{u_{>N}}_{S^*(J_k)}\lesssim_u N^{-s}.
		\end{equation}

For \eqref{bounds on Jk}, we first note that interpolating between \eqref{a priori} and Lemma~\ref{lem:sb} gives $\norm{\vert\nabla\vert^{s_c}u}_{S(J_k)}\lesssim_u 1$. Thus, using the fractional chain rule and Sobolev embedding, we find
	\begin{align*}
	\norm{\vert\nabla\vert^{s_c}u}_{S^*(J_k)}&\lesssim\norm{u}_{L_t^\infty\dot{H}_x^{s_c}(J_k\times\R^3)}+\xnorms{u}{5p/2}{J_k}^p\xnorms{\nsc u}{10/3}{J_k}
	\\ &\lesssim_u 1+\norm{\nsc u}_{S(J_k)}^{p+1}
	\lesssim_u 1.
	\end{align*}

For \eqref{gain on Jk}, we first apply Strichartz and the fractional chain rule to estimate
	\begin{align*}
	\norm{\vert\nabla\vert^s u}_{S(J_k)}&\lesssim \norm{u}_{L_t^\infty\dot{H}_x^s}+\xnorms{u}{5p/2}{J_k}^p\norm{\vert\nabla\vert^s u}_{S(J_k)}.
	\end{align*}
As $\xnorms{u}{5p/2}{J_k}\lesssim_u 1$, a standard bootstrap argument gives 
	$$\norm{\vert\nabla\vert^s u}_{S(J_k)}\lesssim_u 1.$$
Thus, using Bernstein we find
	\begin{align*}
	\norm{u_{>N}}_{S^*(J_k)}&\lesssim N^{-s}\norm{u}_{L_t^\infty\dot{H}_x^{s}(J_k\times\R^3)}+N^{-s}\xnorms{\vert\nabla\vert^s F(u)}{10/7}{J_k}
	\\ &\lesssim_u N^{-s}+N^{-s}\xnorms{u}{5p/2}{J_k}^p\norm{\vert\nabla\vert^s u}_{S(J_k)}
	\lesssim_u N^{-s}.
	\end{align*}
This completes the proof of Corollary~\ref{cor:bilinear}.
\end{proof}


\section{The case $1<s_c<3/2$}\label{section:case one}

In this section, we rule out the existence of solutions as in Theorem~\ref{thm:reduction1} with $1<s_c<3/2$. We treat separately the cases $T_{max}<\infty$ and $T_{max}=\infty$. We show that the existence of solutions with $T_{max}<\infty$ is inconsistent with the conservation of energy. For the case $T_{max}=\infty$, we employ a space-localized Morawetz inequality as in the work of Bourgain on the radial energy-critical NLS \cite{Bou}.

\subsection{Finite time blowup}\label{section:ftbu}
In this section, we preclude the existence of almost periodic solutions as in Theorem~\ref{thm:reduction1} for which $1<s_c<3/2$ and $T_{max}<\infty$.
\begin{theorem}\label{thm:ftbu} There are no almost periodic solutions solutions to \eqref{nls} as in Theorem~\ref{thm:reduction1} with $1<s_c<3/2$ and $T_{max}<\infty$.  
\end{theorem}

\begin{proof}
Suppose $u$ were such a solution.
%
%
For $N>0$ and $t\in[0,T_{max})$, we use Proposition~\ref{reduced duhamel}, Strichartz, H\"older, and Bernstein to estimate
	\begin{align*}
	\norm{\vert\nabla\vert^{\frac{1}{2}s_c} P_Nu(t)}_{L_x^2}&\lesssim\xnorm{\vert\nabla\vert^{\frac12s_c} P_N\big(F(u)\big)}{2}{6/5}{[t,T_{max})}
	\\&\lesssim (T_{max}-t)^{1/2}N^{1-\frac12s_c}\xonorm{\vert\nabla\vert^{\frac12s_c}F(u)}{\infty}{\frac{12p}{11p+4}}
	\\&\lesssim(T_{max}-t)^{1/2}N^{1-\frac12s_c}\xonorm{u}{\infty}{{3p}/{2}}^p\xonorm{\vert\nabla\vert^{\frac12s_c}u}{\infty}{\frac{12p}{3p+4}}.
	\end{align*}
Using Sobolev embedding and \eqref{a priori}, we deduce
	\begin{equation}\label{low}
	\norm{P_{\leq N}\vert\nabla\vert^{\frac12s_c}u(t)}_{L_x^2}\lesssim_u (T_{max}-t)^{1/2}N^{1-\frac12s_c}
	\end{equation}
for $t\in[0,T_{max})$ and $N>0$. 

For the high frequencies, Bernstein and \eqref{a priori} give
	\begin{equation}\label{high}
	\norm{P_{>N}\vert\nabla\vert^{\frac12s_c}u(t)}_{L_x^2}\lesssim N^{-\frac12s_c}\norm{u}_{L_t^\infty\dot{H}_x^{s_c}}\lesssim_u N^{-\frac12s_c}.
	\end{equation}

We now let $\eta>0$. We first choose $N$ sufficiently large so that $N^{-\frac12s_c}<\eta$, and subsequently choose $t$ close enough to $T_{max}$ that $(T_{max}-t)^{1/2}N^{1-\frac12s_c}<\eta$. Adding \eqref{low} and \eqref{high} then gives
	$$\norm{\vert\nabla\vert^{\frac12s_c}u(t)}_{L_x^2}\lesssim_u\eta.$$
As $\eta$ was arbitrary, we deduce
		$$\norm{\vert\nabla\vert^{\frac12s_c}u(t)}_{L_x^2}\to 0\quad\text{as}\quad t\to T_{max}.$$ 
As $\tfrac12s_c<1<s_c$, we may interpolate this with \eqref{a priori} to deduce that 
		\begin{equation}
		\label{kinetic}
		\norm{\nabla u(t)}_{L_x^2}\to 0\quad\text{as}\quad t\to T_{max}.
		\end{equation} 

We can also use interpolation and Sobolev embedding to estimate
		\begin{align*}
		\norm{u(t)}_{L_x^{p+2}}\lesssim\norm{u(t)}_{L_x^6}^{\frac{2}{p+2}}\norm{u(t)}_{L_x^{{3p}/{2}}}^{\frac{p}{p+2}}
				\lesssim \norm{\nabla u(t)}_{L_x^2}^{\frac{2}{p+2}}\norm{\nsc u(t)}_{L_x^2}^{\frac{p}{p+2}},
		\end{align*}
so that \eqref{kinetic} and \eqref{a priori} imply
		\begin{equation}
		\label{potential}
		\norm{u(t)}_{L_x^{p+2}}\to 0\quad\text{as}\quad t\to T_{max}.
		\end{equation}

Adding \eqref{kinetic} and \eqref{potential}, we deduce that $E[u(t)]\to 0$ as $t\to T_{max}$. By the conservation of energy, we conclude $E[u(t)]\equiv 0$. Thus we must have $u\equiv 0$, which contradicts the fact that $u$ blows up. This completes the proof of Theorem~\ref{thm:ftbu}.
\end{proof}


\subsection{Morawetz \`a la Bourgain}\label{section:mor}
	In this section, we preclude the existence of almost periodic solutions as in Theorem~\ref{thm:reduction1} for which $1<s_c<3/2$ and $T_{max}=\infty$. We employ a space-localized version of the Lin--Strauss Morawetz inequality as in the work of Bourgain \cite{Bou} on the radial energy-critical NLS. See also \cite{Gri, Tao}. 

\begin{proposition}[Space-localized Morawetz]\label{prop:mor}
Let $1<s_c<3/2$ and $u:I\times\R^3\to\C$ be a solution to \eqref{nls}. Then for $C\geq1$, we have
	\begin{align}\nonumber
	\int_I\int_{\vert x\vert \leq C\vert I\vert^{1/2}}&\frac{\vert u(t,x)\vert^{p+2}}{\vert x\vert}\,dx\,dt
	\\ &\lesssim (C\vert I\vert^{1/2})^{2s_c-1}
	\bigg\{\norm{u}_{L_t^\infty\dot{H}_x^{s_c}(I\times\R^3)}^2+\norm{u}_{L_t^\infty\dot{H}_x^{s_c}(I\times\R^3)}^{p+2}\bigg\}.\label{eq:mor}
	\end{align}
\end{proposition}

\begin{proof} We follow the presentation in \cite{Tao}. Using the scaling symmetry \eqref{eq:scaling}, we may assume $C\vert I\vert^{1/2}=1$. We define the Morawetz action
	$$\text{Mor}(t):=2\,\Im\int_{\R^3}\nabla a(x)\cdot\nabla u(t,x)\bar{u}(t,x)\,dx,$$
where $a(x)=(\eps^2+\vert x\vert^2)^{1/2}\chi(x),$ $0<\eps\ll 1$, and $\chi$ is a bump function supported on $\{\vert x\vert\leq 2\}$ and equal to one on $\{\vert x\vert\leq 1\}$. 

A standard computation shows that for $u$ solving \eqref{nls}, we have
	\begin{align}
	\partial_t \text{Mor}(t)
	&\geq c\bigg[\int_{\R^3} \Delta a(x)\vert u(t,x)\vert^{p+2}\,dx	\label{main}
	\\ &\quad\quad+\int_{\R^3}-\Delta\Delta a(x)\vert u(t,x)\vert^2\,dx	\label{lot}
	\\ &\quad\quad+\Re\int_{\R^3}\nabla u(t,x)\cdot\nabla^2 a(x)\nabla\bar{u}(t,x)\,dx\bigg]. \label{convex}
	\end{align}

For $\vert x\vert\leq 1$, we have $\Delta a(x)\gtrsim (\eps^2+\vert x\vert^2)^{-1/2}$, while the contributions of $\eqref{lot}$ and $\eqref{convex}$ are nonnegative.  For $1< \vert x\vert\leq 2$, the weight $a$ and its derivatives are bounded uniformly in $\eps$.  Thus, applying the fundamental of calculus, we are led to

	\begin{align}
	\int_I\int_{\vert x\vert\leq 1} \frac{\vert u(t,x)\vert^{p+2}}{(\eps^2+\vert x\vert^2)^{1/2}}\,dx\,dt
	&\lesssim
	\norm{\text{Mor}}_{L_t^\infty(I)} \label{bound M}
	\\ &\quad+
	\int_I\int_{1< \vert x\vert\leq 2}\vert u(t,x)\vert^{p+2}\,dx\,dt\label{bound potential}
	\\ &\quad+
	\int_I\int_{1<\vert x\vert\leq 2}\vert u(t,x)\vert^2\,dx\,dt \label{bound mass}
	\\ &\quad+
	\int_I\int_{1<\vert x\vert\leq 2}\vert \nabla u(t,x)\vert^2\,dx\,dt \label{bound kinetic}
	\end{align}
We now show that \eqref{bound M} through \eqref{bound kinetic} can be controlled by $\norm{u}_{L_t^\infty\dot{H}_x^{s_c}}^2$ or $\norm{u}_{L_t^\infty\dot{H}_x^{s_c}}^{p+2},$ so that sending $\eps\to 0$ we can recover \eqref{eq:mor}. Using H\"older, Sobolev embedding, and the fact that $\vert I\vert=C^{-2}\leq 1$, we estimate	
	\begin{align*}
	&\eqref{bound M}\lesssim\norm{\nabla a}_{L_x^{\frac{3p}{2(p-2)}}(\R^3)}\xnorm{\nabla u}{\infty}{\frac{3p}{p+2}}{I}\xnorm{u}{\infty}{{3p}/{2}}{I} 
	\lesssim \norm{u}_{L_t^\infty\dot{H}_x^{s_c}(I\times\R^3)}^2,
\\	&\eqref{bound potential} \lesssim\vert I\vert\,\norm{1}_{L_x^{\frac{3p}{p-4}}(\vert x\vert\sim 1)}\xnorm{u}{\infty}{3p/2}{I}^{p+2}
	\lesssim \norm{u}_{L_t^\infty\dot{H}_x^{s_c}(I\times\R^3)}^{p+2},
\\	&\eqref{bound mass}\lesssim\vert I\vert\,\norm{1}_{L_x^{\frac{3p}{3p-4}}(\vert x\vert\sim 1)}\xnorm{u}{\infty}{3p/2}{I}^2 \lesssim \norm{u}_{L_t^\infty\dot{H}_x^{s_c}(I\times\R^3)}^2,
\\	&\eqref{bound kinetic}\lesssim \vert I\vert\,\norm{1}_{L_x^{\frac{3p}{p-4}}(\vert x\vert\sim 1)}\xnorm{\nabla u}{\infty}{\frac{3p}{p+2}}{I}^2\lesssim \norm{u}_{L_t^\infty\dot{H}_x^{s_c}(I\times\R^3)}^2.
	\end{align*}	
This completes the proof of Proposition~\ref{prop:mor}.
\end{proof}

We now use Proposition~\ref{prop:mor} and Lemma~\ref{lower bounds} to prove the following.

\begin{theorem}\label{infinite time} There are no almost periodic solutions $u$ to \eqref{nls} as in Theorem~\ref{thm:reduction1} with $1<s_c<3/2$ and $T_{max}=\infty$. 
\end{theorem}

\begin{proof} Suppose that $u$ were such a solution. In particular, $u$ is nonzero and $x(t)\equiv 0$, so that Lemma~\ref{lower bounds} applies. 

We let $I\subset[0,\infty)$ be a compact time interval, which is a contiguous union of characteristic subintervals $J_k$. Choosing $C(u)$ sufficiently large, we use \eqref{a priori}, Proposition~\ref{prop:mor}, H\"older, and \eqref{lb1} to estimate
	\begin{align*}
	\vert I\vert^{s_c-1/2}&
	\gtrsim_u \sum_{J_k\subset I}\int_{J_k}\int_{\vert x\vert\leq C(u)\vert J_k\vert^{1/2}}\frac{\vert u(t,x)\vert^{p+2}}{\vert x\vert}\,dx\,dt
	\\ &\gtrsim_u \sum_{J_k\subset I} \int_{J_k}N_k\int_{\vert x\vert\leq\frac{C(u)}{N_k}}\vert u(t,x)\vert^{p+2}\,dx\,dt
	\\ &\gtrsim_u \sum_{J_k\subset I} \int_{J_k}N_k^{1+\frac{3p}{2}}\bigg(\int_{\vert x\vert\leq\frac{C(u)}{N(t)}}\vert u(t,x)\vert^2\,dx\bigg)^{\frac{p+2}{2}}\,dt
	\\ &\gtrsim_u \sum_{J_k\subset I} \int_{J_k}N(t)^{3+ps_c}(N(t)^{-2s_c})^{\frac{p+2}{2}}\,dt
	\\ &\gtrsim_u \int_I N(t)^{3-2s_c}\,dt.
	\end{align*}
Recalling that $s_c<3/2$ and $\inf N(t)\geq 1$, we now reach a contradiction by taking $I$ sufficiently large inside of $[0,\infty)$. This completes the proof of Theorem~\ref{infinite time}.\end{proof}

			\section{The case $3/4<s_c<1$}\label{section:case two}
In this section, we rule out the existence of solutions as in Theorem~\ref{thm:reduction1} with $3/4<s_c<1$. The main technical tool we use is a long-time Strichartz estimate, Proposition~\ref{prop:LTS2}, which we prove in the next subsection. In Section~\ref{section:rfc2} we rule out the rapid frequency-cascade scenario by showing that the existence such solutions is inconsistent with the conservation of mass. In Section~\ref{section:flmor2} we prove a frequency-localized Lin--Strauss Morawetz inequality, which we use in Section~\ref{section:soliton2} to rule out the quasi-soliton scenario.

The arguments in this section are equally valid in the range $1/2<s_c\leq 3/4$; however, these cases have already been addressed in the non-radial setting \cite{Mur:swamprat}.
			
\subsection{Long-time Strichartz estimates}\label{section:lts2}
In this section, we prove a long-time Strichartz estimate adapted to the Lin--Strauss Morawetz inequality. We will work under the assumption
		\begin{equation}
		\label{eq:inputgain}
		u\in L_t^\infty\dot{H}_x^{s}([0,T_{max})\times\R^3)
		\end{equation}
for some $s\leq s_c$. We have from \eqref{a priori} that \eqref{eq:inputgain} holds for $s=s_c$. In Section~\ref{section:rfc2} we will show that rapid frequency-cascade solutions actually satisfy \eqref{eq:inputgain} for $s<s_c$. 

Throughout Section~\ref{section:case two}, we use the following notation:

		\begin{equation}
		A_I:=\xnorm{\nsc u_{\leq N}}{2}{6}{I},
		\label{def:A2}
		\end{equation}

		\begin{equation}
		K_I:=\int_I N(t)^{3-2s_c}\,dt\sim_u \sum_{J_k\subset I} N_k^{1-2s_c}.
		\label{def:K2}
		\end{equation}

The main result of this section is the following.

\begin{proposition}[Long-time Strichartz estimate]\label{prop:LTS2}
Let $u:[0,T_{max})\times\R^3\to\C$ be an almost periodic solution as in Theorem~\ref{thm:reduction1} with $3/4<s_c<1$. Let $I\subset[0,T_{max})$ be a compact time interval, which is a contiguous union of characteristic subintervals $J_k$. Suppose \eqref{eq:inputgain} holds for some $s_c-1/2<s\leq s_c$. Then for any $N>0$, we have
	\begin{equation}
	\label{eq:LTS2}
	A_I(N)\lesssim_u 1+N^{\sigma(s)}K_I^{1/2},
	\end{equation}
where $\sigma(s):=2s_c-s-1/2.$

In particular, using \eqref{a priori}, we have
	\begin{equation}
	\label{eq:LTS2sc}
	A_I(N)\lesssim_u 1+N^{s_c-1/2}K_I^{1/2}.
	\end{equation}
Moreover, for any $\eps>0$, there exists $N_0=N_0(\eps)>0$ so that for any $N\leq N_0$, 
	\begin{equation}
	\label{eq:LTS2small}
	A_I(N)\lesssim_u \eps(1+N^{s_c-1/2}K_I^{1/2}).
	\end{equation}
\end{proposition}

We prove Proposition~\ref{prop:LTS2} by induction. The inductive step relies on the following. 

\begin{lemma}\label{lem:recurrence2}
Let $\eta>0$ and $u,$ $I,$ $s$, $\sigma$ be as above. For any $N>0$, we have
	\begin{align*}
	\xnorm{\nsc P_{\leq N}\big(F(u)\big)}{2}{6}{I}
	&\lesssim_u C_\eta\sup_{J_k\subset I}\norm{u_{\leq N/\eta}}_{L_t^\infty\dot{H}_x^{s}(J_k\times\R^3)}
	N^{\sigma(s)}K_I^{1/2}
	\\ &\quad\quad +\sum_{M\geq N/\eta}\big(\tfrac{M}{N}\big)^{s_c}A_I(M).
	\end{align*}
\end{lemma}

\begin{proof}
Throughout the proof, all spacetime norms are taken over $I\times\R^3$ unless indicated otherwise.

We fix $0<\eta<1$. Using almost periodicity, we may choose $c(\eta)$ sufficiently small so that
		\begin{equation}
		\label{Ceta22}
		\xnorm{\nsc u_{\leq c(\eta)N(t)}}{\infty}{2}{I}<\eta.
		\end{equation}

We decompose the nonlinearity as follows:
				$$F(u)=F(u_{\leq N/\eta})+[F(u)-F(u_{\leq N/\eta})].$$
				
We first restrict our attention to an individual characteristic subinterval $J_k$. Using the fractional chain rule, H\"older, the triangle inequality, and Sobolev embedding, we estimate
	\begin{align}
	\xnorm{\nsc &P_{\leq N} F(u_{\leq N/\eta})}{2}{6/5}{J_k}\nonumber
	\\ &\lesssim\xnorm{u_{\leq N/\eta}}{\infty}{3p/2}{J_k}^p\xnorm{\nsc u_{\leq N/\eta}}{2}{6}{J_k}\nonumber
	\\ &\lesssim \xnorm{\nsc P_{\leq c(\eta)N_k}u_{\leq N/\eta}}{\infty}{2}{J_k}^p\xnorm{\nsc u_{\leq N/\eta}}{2}{6}{J_k}\nonumber
		\\ &\quad +\xnorm{\nsc P_{>c(\eta)N_k}u_{\leq N/\eta}}{\infty}{2}{J_k}^p\xnorm{\nsc u_{\leq N/\eta}}{2}{6}{J_k}.\nonumber
		\end{align}

For the first term, we use \eqref{Ceta22} to get
	\begin{align}\xnorm{\nsc& P_{\leq c(\eta)N_k}u_{\leq N/\eta}}{\infty}{2}{J_k}^p\xnorm{\nsc u_{\leq N/\eta}}{2}{6}{J_k}\nonumber
	\\ &\label{lts absorb 2}\lesssim \eta^{s_c}\xnorm{\nsc u_{\leq N/\eta}}{2}{6}{J_k}.\end{align}

For the next term, we note that we only need to consider the case $c(\eta)N_k< N/\eta$, in which case we have $1\lesssim C_\eta(\tfrac{N}{N_k})^{s_c-1/2}.$ Using Bernstein, Lemma~\ref{lem:sb}, and \eqref{a priori}, we estimate
	\begin{align}
	\xnorm{\nsc& P_{>c(\eta)N_k}u_{\leq N/\eta}}{\infty}{2}{J_k}^p\xnorm{\nsc u_{\leq N/\eta}}{2}{6}{J_k}\nonumber
	\\ &\lesssim_u C_\eta(\tfrac{N}{N_k})^{s_c-1/2}\xnorm{\nsc u_{\leq N/\eta}}{\infty}{2}{J_k}\nonumber
	\\ &\lesssim_u C_\eta(\tfrac{N}{N_k})^{s_c-1/2}N^{s_c-s}\norm{u_{\leq N/\eta}}_{L_t^\infty\dot{H}_x^{s}(J_k\times\R^3)}.\label{lts gain 2}
	\end{align}
Summing \eqref{lts absorb 2} and \eqref{lts gain 2} over $J_k\subset I$ and using \eqref{def:K2}, we find
	\begin{align}
	\xonorm{\nsc& P_{\leq N}F(u_{\leq N/\eta})}{2}{6/5}\nonumber
	\\ &\lesssim_u \eta^{s_c} A_I(N/\eta)\nonumber
	+ C_\eta \sup_{J_k\subset I}\norm{u_{\leq N/\eta}}_{L_t^\infty\dot{H}_x^s(J_k\times\R^3)} 
	N^{\sigma(s)}K_I^{1/2}.\nonumber
	\end{align}
	
Next we use Bernstein, H\"older, Sobolev embedding, and \eqref{a priori} to estimate 
		\begin{align*}
		\xonorm{\nsc P_{\leq N}\big(F(u)-F(u_{\leq N/\eta})\big)}{2}{6/5}&\lesssim
		N^{s_c}\xonorm{u}{\infty}{3p/2}^p
		\sum_{M>N/\eta}\xonorm{u_M}{2}{6}
		\\&\lesssim_u
		\sum_{M>N/\eta}\big(\tfrac{N}{M}\big)^{s_c}A_I(M).
		\end{align*}

Collecting the estimates, we complete the proof of Lemma~\ref{lem:recurrence2}. 
\end{proof}

We turn to the proof of Proposition~\ref{prop:LTS2}.

\begin{proof}[Proof of Proposition~\ref{prop:LTS2}] We proceed by induction. For the base case, we take $N>\sup_{t\in I}N(t)\geq 1$, so that $N^{2(s_c-s)}\big(\tfrac{N}{N(t)}\big)^{2s_c-1}\geq 1$ for $t\in I$. Lemma~\ref{lem:sb} gives
	\begin{align*}
	A_I(N)^2\lesssim_u 1+\smallint_I N(t)^2\,dt \lesssim_u N^{2(s_c-s)}N^{2s_c-1}K_I.
	\end{align*}
Thus for $N>\sup_{t\in I}N(t)$, we have
	\begin{equation}\label{basecase2}
	A_I(N)\leq C_u\big[1+N^{\sigma(s)}K_I^{1/2}\big].
	\end{equation}
This inequality clearly remains true if we replace $C_u$ by any larger constant.

We now suppose that \eqref{basecase2} holds at frequencies $\geq 2N$; we will use Lemma~\ref{lem:recurrence2} to show that it holds at frequency $N$.

Applying Strichartz, Lemma~\ref{lem:recurrence2}, \eqref{a priori}, and \eqref{eq:input22} gives
	\begin{align}
	A_I(N)\nonumber
	&\leq\tilde{C}_u\nonumber
	\big[\inf_{t\in I}\norm{u_{\leq N}(t)}_{\dot{H}_x^{s_c}}
	+C_\eta\sup_{J_k\subset I}
	\norm{u_{\leq N/\eta}}_{L_t^\infty\dot{H}_x^s(J_k\times\R^3)} 
	N^{\sigma(s)}K_I^{1/2}
	\\&\quad\quad\quad\quad+\sum_{M\geq N/\eta}\big(\tfrac{N}{M}\big)^{s_c}A_I(M)\big]\label{for smallness2}
	\\ &\leq\tilde{C}_u\big[1+C_\eta N^{-\sigma(s)}K_I^{1/2}+\sum_{M\geq N/\eta}\big(\tfrac{N}{M}\big)^{s_c}A_I(M)\big].\nonumber
	\end{align}
We let $\eta<1/2$ and notice that $s>s_c-1/2$ guarantees $\sigma(s)<s_c$. Thus, using the inductive hypothesis, we find
	\begin{align*}
	A_I(N)
	&\leq \tilde{C}_u\big[1+C_\eta N^{\sigma(s)}K_I^{1/2}+\sum_{M\geq N/\eta}\big(\tfrac{N}{M})^{s_c}(C_u+C_uM^{\sigma(s)}K_I^{1/2})\big]
	\\ &\leq \tilde{C}_u\big[1+C_\eta N^{\sigma(s)}K_I^{1/2}\big]+C_u\tilde{C}_u\big[\eta^{s_c}+\eta^{s_c-\sigma(s)}N^{\sigma(s)}K_I^{1/2}\big]
	\end{align*}
Choosing $\eta$ sufficiently small depending on $\tilde{C}_u$, we find
	$$A_I(N)\leq\tilde{C}_u(1+C_\eta N^{\sigma(s)}K_I^{1/2})
	+\tfrac12 C_u(1+N^{\sigma(s)}K_I^{1/2}).$$
Finally, choosing $C_u$ possibly even larger to guarantee $C_u\geq 2(1+C_\eta)\tilde{C}_u$, we deduce from the above inequality that
	$$A_I(N)\leq C_u(1+N^{\sigma(s)}K_I^{1/2}),$$
as was needed to show. This completes the proof of \eqref{eq:LTS2}.

The estimate \eqref{eq:LTS2sc} follows directly from \eqref{eq:LTS2} with $s=s_c$. With \eqref{eq:LTS2sc} in place, we can prove \eqref{eq:LTS2small} by continuing from \eqref{for smallness2}, choosing $\eta$ sufficiently small, and noting that $\inf_{t\in I}N(t)\geq 1$ implies
	$$\lim_{N\to 0}\big[\inf_{t\in I}\norm{u_{\leq N}(t)}_{\dot{H}_x^{s_c}} 
	+\sup_{J_k\subset I}\norm{u_{\leq N/\eta}}_{L_t^\infty\dot{H}_x^{s_c}(J_k\times\R^3)}\big]=0$$
for any $\eta>0$. This completes the proof of Proposition~\ref{prop:LTS2}.
\end{proof}


\subsection{The rapid frequency-cascade scenario}\label{section:rfc2} In this section we rule out solutions as in Theorem~\ref{thm:reduction1} for which $3/4<s_c<1$ and 
	\begin{equation}
	\label{eq:finite K2}
	K_{[0,T_{max})}=\int_0^{T_{max}} N(t)^{3-2s_c}\,dt<\infty.
	\end{equation}
We show that \eqref{eq:finite K2} and Proposition~\ref{prop:LTS2} imply that such solutions would possess additional decay. We then use the conservation of mass to derive a contradiction.

	Note that we have
	\begin{equation}
	\label{N to infinity}
	\lim_{t\to T_{max}}N(t)=\infty,
	\end{equation}
whether $T_{max}$ is finite or infinite. Indeed, in the case $T_{max}<\infty$ this follows from Corollary~\ref{cor:constancy}, while in the case $T_{max}=\infty$ this follows from \eqref{eq:finite K2} and \eqref{def:K2}. 

We begin with the following lemma. 

\begin{lemma}[Improved decay]\label{lem:improved decay} Let $u:[0,T_{max})\times\R^3\to\C$ be an almost periodic solution as in Theorem~\ref{thm:reduction1} with $3/4<s_c<1$. Suppose
		\begin{equation}
		\label{eq:input22}
		u\in L_t^\infty\dot{H}_x^{s}([0,T_{max})\times\R^3)\quad\text{for some}\quad s_c-1/2< s\leq s_c.
		\end{equation}
If \eqref{eq:finite K2} holds, then
		\begin{equation}
		\label{eq:output22}
		u\in L_t^\infty\dot{H}_x^{\sigma}([0,T_{max})\times\R^3)\quad\text{for all}\quad s-\sigma(s)<\sigma\leq s_c,
		\end{equation}
where as above $\sigma(s):=2s_c-s-1/2$.
\end{lemma}

\begin{proof} Throughout the proof, we take all spacetime norms over $[0,T_{max})\times\R^3.$ 

We first use Proposition~\ref{prop:LTS2} and \eqref{eq:finite K2} to show 
		\begin{equation}
		\label{eq:quantA}
		A_{[0,T_{max})}(N)\lesssim_u N^{\sigma(s)}.
		\end{equation}

Let $I_n\subset [0,T_{max})$ be a nested sequence of compact time intervals, each of which is a contiguous union of chracteristic subintervals. We let $\eta>0$ and apply Strichartz, Lemma~\ref{lem:recurrence2}, and \eqref{eq:input22} to estimate
	$$A_{I_n}(N)\lesssim_u \inf_{t\in I_n}\norm{u_{\leq N}(t)}_{\dot{H}_x^{s_c}}+C_\eta N^{\sigma(s)}K_{I_n}^{1/2}+\sum_{M\geq N/\eta}\big(\tfrac{N}{M}\big)^{s_c}A_{I_n}(M).$$ 
As \eqref{eq:LTS2} gives $A_{I_n}(N)\lesssim_u 1+N^{\sigma(s)}K_{I_n}^{1/2}$, we may choose $\eta$ sufficiently small and continue from above to get
	\begin{equation}
	\label{eq:idmid}
	A_{I_n}(N)\lesssim_u \inf_{t\in I_n}\norm{u_{\leq N}}_{\dot{H}_x^{s_c}}+N^{\sigma(s)}K_{I_n}^{1/2}.
	\end{equation}

Using \eqref{N to infinity}, we get that for any $N>0$ we have $\lim_{t\to T_{max}}\norm{u_{\leq N}}_{\dot{H}_x^{s_c}}=0.$ Thus sending $n\to\infty$, continuing from \eqref{eq:idmid}, and using \eqref{eq:finite K2}, we deduce that \eqref{eq:quantA} holds.

We next show that \eqref{eq:quantA} and \eqref{eq:input22} imply 
		\begin{equation}
		\label{eq:quantmass}
		\xonorm{\vert\nabla\vert^{s}u_{\leq N}}{\infty}{2}\lesssim_u N^{\sigma(s)}.
		\end{equation}

We first use Proposition~\ref{reduced duhamel} and Strichartz to estimate
	$$\xonorm{\vert\nabla\vert^s u_{\leq N}}{\infty}{2}\lesssim_u
	\xonorm{\vert\nabla\vert^s P_{\leq N}\big(F(u)\big)}{2}{6/5}.$$
	
We decompose the nonlinearity as $F(u)=F(u_{\leq N})+[F(u)-F(u_{\leq N})].$ Noting that $s_c-1/2<s\leq s_c$ implies $6\leq \tfrac{3p}{2+ps-p}<\infty$, we can first use H\"older, the fractional chain rule, Sobolev embedding, \eqref{a priori}, and \eqref{eq:input22} to estimate
	\begin{align*}
	\xonorm{\vert\nabla\vert^s F(u_{\leq N})}{2}{6/5}
	&\lesssim\xonorm{u}{\infty}{3p/2}^{p-1}\xonorm{u}{\infty}{\frac{6}{3-2s}}\xonorm{\vert\nabla\vert^s u_{\leq N}}{2}{\frac{3p}{2+ps-p}}
	\\ &\lesssim \xonorm{\nsc u}{\infty}{3p/2}^{p-1}\xonorm{\vert\nabla\vert^s u}{\infty}{2}
	\xonorm{\nsc u_{\leq N}}{2}{6}
	\\ &\lesssim_u N^{\sigma(s)}.
	\end{align*}

Next, we note that $F(u)-F(u_{\leq N})=\text{\O}(u_{>N}u^p)$ and that $s>s_c-1/2$ implies $\sigma(s)<s_c$. Thus we can use Bernstein, Lemma~\ref{lem:paraproduct}, the fractional chain rule, Sobolev embedding, \eqref{a priori}, \eqref{eq:input22}, and \eqref{eq:quantA} to estimate
	\begin{align}
	\xonorm{&\vert\nabla\vert^s P_{\leq N}\big(F(u)-F(u_{\leq N})\big)}{2}{6/5}
	\nonumber
	\\ &\lesssim N^{s_c}\xonorm{\vert\nabla\vert^{-(s_c-s)}\big(u^pu_{>N}\big)}{2}{6/5}\nonumber
	\\ &\lesssim N^{s_c}\xonorm{\vert\nabla\vert^{s_c-s}(u^p)}{\infty}{\frac{3p}{5p-4-2ps}}\sum_{M>N}\xonorm{\vert\nabla\vert^{-(s_c-s)}u_M}{2}{\frac{3p}{2+ps-p}}\nonumber
	\\ &\lesssim \xonorm{u}{\infty}{3p/2}^{p-2}\xonorm{u}{\infty}{\frac{6}{3-2s}}\xonorm{\vert\nabla\vert^{s_c-s}u}{\infty}{\frac{6}{3-2s}}\nonumber
	\sum_{M>N}\big(\tfrac{N}{M}\big)^{s_c}\xonorm{\vert\nabla\vert^{s}u_M}{2}{\frac{3p}{2+ps-p}}
	\\&\lesssim\xonorm{\vert\nabla\vert^{s_c}u}{\infty}{2}^{p-1}\xonorm{\vert\nabla\vert^s u}{\infty}{2}\sum_{M>N}\big(\tfrac{N}{M}\big)^{s_c}\xonorm{\nsc u_M}{2}{6}
	\nonumber
	\\ &\lesssim_u\nonumber\sum_{M>N}\big(\tfrac{N}{M}\big)^{s_c} M^{\sigma(s)}.
	\\&\lesssim_uN^{\sigma(s)}.\nonumber
	\end{align}
Note that in the case $s=s_c$, we would simply use H\"older instead of Lemma~\ref{lem:paraproduct} and the fractional chain rule.

%
The last two estimates together imply \eqref{eq:quantmass}. 

Finally, we use \eqref{eq:quantmass} to prove \eqref{eq:output22}. We fix $s-\sigma(s)<\sigma\leq s_c$ and use Bernstein, \eqref{a priori}, and \eqref{eq:quantmass} to estimate
	\begin{align*}
	\xonorm{\vert\nabla\vert^{\sigma}u}{\infty}{2}
	&\lesssim
	\xonorm{\nsc u_{\geq 1}}{\infty}{2}+
	\sum_{M\leq 1}M^{\sigma-s}\xonorm{\vert\nabla\vert^s u_{M}}{\infty}{2}
	\\ &\lesssim_u 1+\sum_{M\leq 1} M^{\sigma-s+\sigma(s)}
	\lesssim_u 1.
	\end{align*}
This completes the proof of Lemma~\ref{lem:improved decay}. \end{proof}

We now iterate Lemma~\ref{lem:improved decay} to establish additional decay.

\begin{proposition}[Additional decay]\label{prop:additional decay} Let $u:[0,T_{max})\times\R^3\to\C$ be an almost periodic solution as in Theorem~\ref{thm:reduction1} with $3/4<s_c<1$. If \eqref{eq:finite K2} holds, then $u\in L_t^\infty\dot{H}_x^{-\eps}$ for some $\eps>0$. 
\end{proposition}

\begin{proof} Let $0<\delta<\min\{1/4,s_c-1/2\}$ and for each $n\geq 0$ define $s_n:=s_c-n\delta$. We have from Lemma~\ref{lem:improved decay} that
	$$u\in L_t^\infty\dot{H}_x^{s_n}\implies u\in L_t^\infty\dot{H}_x^{\sigma}\quad\text{for all}\quad 0\leq n<\tfrac{1}{2\delta}\quad\text{and}\quad s_n-\sigma(s_n)<\sigma\leq s_c.$$
The restriction $n<\tfrac{1}{2\delta}$ guarantees $s_n>s_c-1/2$. As above, $\sigma(s):=2s_c-s-1/2$.

As \eqref{a priori} gives $u\in L_t^\infty\dot{H}_x^{s_0}$ and the constraint $0<\delta<s_c-1/2$ guarantees $s_n-\sigma(s_n)<s_{n+1}\leq s_c$ for all $n\geq 0$, we get by induction that 
	\begin{equation}
	\label{eq:adhelp}
	u\in L_t^\infty\dot{H}_x^{\sigma}\quad\text{for all}
	\quad 0\leq n<\tfrac{1}{2\delta}\quad\text{and}
	\quad s_n-\sigma(s_n)<\sigma\leq s_c.
	\end{equation}

As $\delta<1/4$, we may find $n^*$ so that $\tfrac{1}{4\delta}<n^*<\tfrac{1}{2\delta}.$ As the constraint $n^*>\tfrac{1}{4\delta}$ implies $s_{n^*}-\sigma(s_{n^*})<0$, we deduce from \eqref{eq:adhelp} that $u\in L_t^\infty\dot{H}_x^{-\eps}$ for some $\eps>0$. This completes the proof of Proposition~\ref{prop:additional decay}. \end{proof}

Finally, we turn to the following.

\begin{theorem}[No frequency-cascades]\label{thm:rfc2} There are no almost periodic solutions as in Theorem~\ref{thm:reduction1} with $3/4<s_c<1$ such that \eqref{eq:finite K2} holds.
\end{theorem}

\begin{proof} Suppose $u$ were such a solution and let $\eta>0$. By almost periodicity, we may find $c(\eta)$ small enough that $\xonorm{\nsc u_{\leq c(\eta)N(t)}}{\infty}{2}<\eta.$ Thus, by interpolation and Proposition~\ref{prop:additional decay}, we have
	\begin{align*}
	\xonorm{u_{\leq c(\eta)N(t)}}{\infty}{2}\lesssim
	\xonorm{\nsc u_{\leq c(\eta)N(t)}}{\infty}{2}^{\frac{\eps}{s_c+\eps}}
	\xonorm{\vert\nabla\vert^{-\eps}u}{\infty}{2}^{\frac{s_c}{s_c+\eps}}
	\lesssim_u \eta^{\frac{\eps}{s_c+\eps}}
	\end{align*}
for some $\eps>0$. 

On the other hand, using Bernstein and \eqref{a priori} we get
	$$\norm{u_{>c(\eta)N(t)}(t)}_{L_x^2}
	\lesssim_u [c(\eta)N(t)]^{-s_c}
	\quad\text{for any}
	\quad t\in[0,T_{max}).$$
	
Thus
	$$\norm{u(t)}_{L_x^2}\lesssim_u \eta^{\frac{\eps}{s_c+\eps}}+[c(\eta)N(t)]^{-s_c}\quad\text{for any}\quad t\in[0,T_{max}).$$

Using \eqref{N to infinity} and the fact that $\eta>0$ was arbitrary, we deduce
	$$\norm{u(t)}_{L_x^2}\to 0\quad\text{as}\quad t\to T_{max}.$$
By the conservation of mass, we conclude that $M[u(t)]\equiv 0$. Thus we must have that $u\equiv 0$, which contradicts that $u$ blows up. This completes the proof of Theorem~\ref{thm:rfc2}. \end{proof}


\subsection{A frequency-localized Lin--Strauss Morawetz inequality}\label{section:flmor2}

In this section, we use Proposition~\ref{prop:LTS2} to prove a frequency-localized Lin--Strauss Morawetz inequality, which we use in Section~\ref{section:soliton2} to rule out the quasi-soliton scenario. As $s_c>1/2$, we prove an estimate that is localized to high frequencies. 

The main result of this section is the following.

\begin{proposition}[Frequency-localized Morawetz]\label{prop:flmor2} Let $u:[0,T_{max})\times\R^3\to\C$ be an almost periodic solution as in Theorem~\ref{thm:reduction1} with $3/4<s_c<1$. Let $I\subset[0,T_{max})$ be a compact time interval, which is a contiguous union of characteristic subintervals $J_k$. Then for any $\eta>0$, there exists $N_0=N_0(\eta)>0$ such that for $N<N_0$, we have
\begin{equation}\label{eq:flmor2}
\iint_{I\times\R^3}\frac{\vert u_{>N}(t,x)\vert^{p+2}}{\vert x\vert}\,dx\,dt\lesssim_u\eta(N^{1-2s_c}+K_I),
\end{equation}
where $K_I$ is as in \eqref{def:K2}.
\end{proposition}

To prove Proposition~\ref{prop:flmor2}, we begin as in the proof of the standard Lin--Strauss Morawetz inequality \eqref{eq:ls}. We truncate the low frequencies of the solution and work with $u_{>N}$ for some $N>0$. As $u_{>N}$ is not a true solution to \eqref{nls}, we need to control error terms arising from this frequency projection. To do this, we choose $N$ small enough to capture `most' of the solution and use the estimates proved in Section~\ref{section:lts2}. We make these notions precise in the following lemma. 

\begin{lemma}[High and low frequency control]\label{lem:highlo} Let $u,$ $I,$ $K_I$ be as above. With all spacetime norms over $I\times\R^3$, we have the following.

For any $N>0$, we have
	\begin{equation}\label{2high}
	\xonorm{u_{>N}}{2}{6}\lesssim_u N^{-s_c}(1+N^{2s_c-1}K_I)^{1/2}.
	\end{equation}

For any $\eta>0$, there exists $N_1=N_1(\eta)$ so that for $N<N_1$, we have
	\begin{equation}\label{2highsmall}
	\xonorm{\vert\nabla\vert^{1/2}u_{>N}}{\infty}{2}\lesssim_u \eta N^{1/2-s_c}.
	\end{equation}

For any $\eta>0$, there exists $N_2=N_2(\eta)$ so that for $N<N_2$, we have
	\begin{equation}\label{2low}
	\xonorm{\nsc u_{\leq N}}{2}{6}\lesssim_u \eta(1+N^{2s_c-1}K_I)^{1/2}.
	\end{equation}
\end{lemma}

\begin{proof} For \eqref{2high}, we use Bernstein and \eqref{eq:LTS2sc} to estimate
	\begin{align*}
	\xonorm{u_{>N}}{2}{6}&\lesssim \sum_{M>N} M^{-s_c}\xonorm{\nsc u_{M}}{2}{6}
	\\&\lesssim_u\sum_{M>N} M^{-s_c}(1+M^{2s_c-1}K_I)^{1/2}
	\\&\lesssim_u N^{-s_c}(1+N^{2s_c-1}K_I)^{1/2}.
	\end{align*}
	
For \eqref{2highsmall}, we let $\eta>0$. Using almost periodicity and the fact that $\inf N(t)\geq 1$, we may find $c(\eta)>0$ so that $\xonorm{\nsc u_{\leq c(\eta)}}{\infty}{2}<\eta.$ Thus Bernstein gives
	\begin{align*}
	\xonorm{\vert\nabla\vert^{1/2}&u_{>N}}{\infty}{2}
	\\ &\lesssim c(\eta)^{1/2-s_c}\xonorm{\nsc u_{>c(\eta)}}{\infty}{2}
	+N^{1/2-s_c}\xonorm{\nsc u_{N<\cdot\leq c(\eta)}}{\infty}{2}
	\\ &\lesssim_u c(\eta)^{1/2-s_c}+\eta N^{1/2-s_c}.
	\end{align*}
Choosing $N_1\ll \eta^{1/(s_c-1/2)}c(\eta)$, we recover \eqref{2highsmall}.

Finally, we note that \eqref{2low} is just a restatement of \eqref{eq:LTS2small}.
\end{proof}

We turn to the proof of Proposition~\ref{prop:flmor2}.

\begin{proof}[Proof of Proposition~\ref{prop:flmor2}] Throughout the proof, we take all spacetime norms over $I\times\R^3$.

We let $0<\eta\ll1$ and choose
	$$N<\min\{N_1(\eta), \eta^2 N_2(\eta^{2s_c})\},$$
where $N_1$ and $N_2$ are as in Lemma~\ref{lem:highlo}. In particular, we note that \eqref{2high} gives
	\begin{equation}\label{still small1}
	\xonorm{u_{>N/\eta^2}}{2}{6}\lesssim_u \eta N^{-s_c}(1+N^{2s_c-1}K_I)^{1/2}.
	\end{equation}
Moreover, as $N/\eta^2< N_2(\eta^{2s_c})$, we can apply \eqref{2low} to get
	\begin{equation}\label{still small2}
	\xonorm{\nsc u_{\leq N/\eta^2}}{2}{6}\lesssim_u \eta(1+N^{2s_c-1}K_I)^{1/2}.
	\end{equation}

We define the Morawetz action 
$$\text{Mor}(t)=2\,\Im\int_{\R^3}\frac{x}{\vert x\vert}\cdot\nabla u_{>N}(t,x)\bar{u}_{>N}(t,x)\,dx.$$
A standard computation using $(i\partial_t+\Delta)u_{>N}=P_{>N}\big(F(u)\big)$ gives
	\begin{equation}
	\nonumber
	\partial_t \text{Mor}(t)\gtrsim\int_{\R^3} \frac{x}{\vert x\vert}\cdot\{P_{>N}\big(F(u)\big),u_{>N}\}_P\,dx,
	\end{equation}
where the \emph{momentum bracket} $\{\cdot,\cdot\}_P$ is defined by $\{f,g\}_P:=\Re(f\nabla\bar{g}-g\nabla\bar{f}).$
Thus, by the fundamental theorem of calculus, we get
\begin{align}
\iint_{I\times\R^3}\frac{x}{\vert x\vert}\cdot\{P_{> N}\big(F(u)\big),u_{> N}\}_P\,dx\lesssim\norm{\text{Mor}}_{L_t^\infty(I)}.\label{flmor start}
\end{align}	
	
Noting that $\{F(u),u\}_P=-\tfrac{p}{p+2}\nabla(\vert u\vert^{p+2}),$ we may write
	\begin{align*}
	\{P_{>N}&\big(F(u)\big),u_{>N}\}_P
	\\&=\{F(u),u\}_P-\{F(u_{\leq N}),u_{\leq N}\}_P
	\\ &\quad-\{F(u)-F(u_{\leq N}),u_{\leq N}\}_P-\{P_{\leq N}\big(F(u)\big),u_{>N}\}_P
	\\ &=-\tfrac{p}{p+2}\nabla(\vert u\vert^{p+2}-\vert u_{\leq N}\vert^{p+2})
	-\{F(u)-F(u_{\leq N}),u_{\leq N}\}_P
	\\&\quad-\{P_{\leq N}\big(F(u)\big),u_{>N}\}_P
	\\&=:I+II+III.
	\end{align*}

Integrating by parts, we see that $I$ contributes to the left-hand side of \eqref{flmor start} a multiple of
	\begin{align*}
	\iint_{I\times\R^3}&\frac{\vert u_{>N}(t,x)\vert^{p+2}}{\vert x\vert}\,dx\,dt
	\end{align*}
and to the right-hand side of \eqref{flmor start} a multiple of
	\begin{equation}
	\xonorms{\tfrac{1}{\vert x\vert}(\vert u\vert^{p+2}-\vert u_{>N}\vert^{p+2}-\vert u_{\leq N}\vert^{p+2})}{1}. \label{i}
	\end{equation}
	
For term $II$, we use $\{f,g\}_P=\nabla \text\O(fg)+\text{\O}(f\nabla g)$. When the derivative hits the product, we integrate by parts. We find that $II$ contributes to the right-hand side of \eqref{flmor start} a multiple of
	\begin{align}
	&\xonorms{\tfrac{1}{\vert x\vert}u_{\leq N}[F(u)-F(u_{\leq N})]}{1} \label{ii}
	\\ &+\xonorms{\nabla u_{\leq N}[F(u)-F(u_{\leq N})]}{1}. \label{iii}
	\end{align}
	
Finally, for $III$, we integrate by parts when the derivative hits $u_{>N}.$ We find that $III$ contributes to the right-hand side of \eqref{flmor start} a multiple of
	\begin{align}
	&\xonorms{\tfrac{1}{\vert x\vert}u_{>N}P_{\leq N}\big(F(u)\big)}{1} \label{iv}
	\\&+\xonorms{u_{>N}\nabla P_{\leq N}\big(F(u)\big)}{1}. \label{v}
	\end{align}

Thus, continuing from \eqref{flmor start}, we see that to complete the proof of Proposition~\ref{prop:flmor2} it will suffice to show that
	\begin{equation}\label{btmq}
	\norm{\text{Mor}}_{L_t^\infty(I)}\lesssim_u \eta N^{1-2s_c}
	\end{equation}
and that the error terms \eqref{i} through \eqref{v} are acceptable, in the sense that they can be controlled by $\eta(N^{1-2s_c}+K_I).$ 

To prove \eqref{btmq}, we use Bernstein, \eqref{hardys}, \eqref{2highsmall} to estimate
	\begin{align*}
	\norm{\text{Mor}}_{L_t^\infty(I)}&\lesssim\xonorm{\vert\nabla\vert^{-1/2}\nabla u_{>N}}{\infty}{2}\xonorm{\vert\nabla\vert^{1/2}(\tfrac{x}{\vert x\vert}u_{>N})}{\infty}{2}
	\\ &\lesssim\xonorm{\vert\nabla\vert^{1/2}u_{>N}}{\infty}{2}^2
	\lesssim_u \eta N^{1-2s_c}.
	\end{align*}
	 
We next turn to the estimation of the error terms \eqref{i} through \eqref{v}.

For \eqref{i}, we first write
	\begin{align}
	\eqref{i}&\lesssim \xonorms{\tfrac{1}{\vert x\vert}(u_{\leq N})^{p+1}u_{>N}}{1}\label{ib1}
	\\ &\quad+\xonorms{\tfrac{1}{\vert x\vert}u_{\leq N}(u_{>N})^{p+1}}{1}.\label{ib2}
	\end{align}

For \eqref{ib1}, we use H\"older, Hardy, the chain rule, Bernstein, \eqref{a priori}, \eqref{2high}, and \eqref{2low} to estimate
	\begin{align*}
	\xonorms{\tfrac{1}{\vert x\vert}(u_{\leq N})^{p+1}u_{>N}}{1}
	&\lesssim\xonorm{\tfrac{1}{\vert x\vert}(u_{\leq N})^{p+1}}{2}{6/5}
		\xonorm{u_{>N}}{2}{6}
	\\ &\lesssim \xonorm{\nabla(u_{\leq N})^{p+1}}{2}{6/5}
		\xonorm{u_{>N}}{2}{6}
	\\ &\lesssim\xonorm{u}{\infty}{3p/2}^p\xonorm{\nabla u_{\leq N}}{2}{6}
	\xonorm{u_{>N}}{2}{6}
	\\ &\lesssim_u \eta N^{1-2s_c}(1+N^{2s_c-1}K_I),
	\end{align*}
which is acceptable.

For \eqref{ib2}, we consider two cases. If $\vert u_{\leq N}\vert\ll \vert u_{>N}\vert$, then we can absorb this term into the left-hand side of \eqref{flmor start}, provided we can show 
	\begin{equation}\label{is it finite}
	\xonorms{\tfrac{1}{\vert x\vert}\vert u_{>N}\vert^{p+2}}{1}<\infty.
	\end{equation}
Otherwise, we are back in the situation of \eqref{ib1}, which we have already handled. Thus, to render \eqref{ib2} an acceptable error term it suffices to establish \eqref{is it finite}. To this end, we use Hardy, Sobolev embedding, Bernstein, and Lemma~\ref{lem:sb} to estimate
	\begin{align*} 
	\xonorms{\tfrac{1}{\vert x\vert} \vert u_{>N}\vert^{p+2}}{1}
	&\lesssim \xonorms{\vert x\vert^{-\frac{1}{p+2}}u_{>N}}{p+2}^{p+2}
	\lesssim\xonorms{\vert\nabla\vert^{\frac{1}{p+2}}u_{>N}}{p+2}^{p+2}
	\\ &\lesssim\xonorm{\vert\nabla\vert^{\frac{3p-2}{2(p+2)}}u_{>N}}{p+2}{\frac{6(p+2)}{3p+2}}^{p+2}
	 \lesssim N^{1-2s_c}\xonorm{\vert\nabla\vert^{s_c} u}{p+2}{\frac{6(p+2)}{3p+2}}^{p+2}
	\\ &\lesssim_u N^{1-2s_c}(1+\smallint_I N(t)^{2}\,dt)<\infty.
	\end{align*}

We next turn to \eqref{ii}. Writing
	\begin{align*}
	\xonorms{\tfrac{1}{\vert x\vert}u_{\leq N}[F(u)-F(u_{\leq N})]}{1}
	&\lesssim\xonorms{\tfrac{1}{\vert x\vert}(u_{\leq N})^{p+1}u_{>N}}{1}
	\\&\quad+\xonorms{\tfrac{1}{\vert x\vert}u_{\leq N}(u_{>N})^{p+1}}{1},
	\end{align*}
we recognize the error terms that we just estimated, namely \eqref{ib1} and \eqref{ib2}. Thus \eqref{ii} is acceptable.

For \eqref{iii}, we use H\"older, \eqref{a priori} \eqref{2high}, and \eqref{2low} to estimate
	\begin{align*}
	\eqref{iii}&\lesssim\xonorm{\nabla u_{\leq N}}{2}{6}
	\xonorm{u_{>N}}{2}{6}
	\xonorm{u}{\infty}{3p/2}^p
	\\&\lesssim \eta N^{1-2s_c}(1+N^{2s_c-1}K_I),
	\end{align*}
which is acceptable.

Finally, for \eqref{iv} and \eqref{v}, we first use Hardy to estimate
	\begin{align*}
	\eqref{iv}&+\eqref{v}
	\\&\lesssim\xonorm{u_{>N}}{2}{6}\xonorm{\tfrac{1}{\vert x\vert} P_{\leq N}\big( F(u)\big)}{2}{6/5}+\xonorm{u_{>N}}{2}{6}\xonorm{\nabla P_{\leq N}\big(F(u)\big)}{2}{6/5}
	\\&\lesssim\xonorm{u_{>N}}{2}{6}\xonorm{\nabla P_{\leq N}\big(F(u)\big)}{2}{6/5}
	\end{align*}

Thus, in light of \eqref{2high} it suffices to prove
	\begin{equation}
	\nonumber
	\xonorm{\nabla P_{\leq N}\big(F(u)\big)}{2}{6/5}\lesssim_u \eta N^{1-s_c}(1+N^{2s_c-1}K_I)^{1/2}.
	\end{equation}

To this end, we use H\"older, Bernstein, the fractional chain rule, \eqref{a priori}, \eqref{still small1}, and \eqref{still small2} to estimate
	\begin{align*}
	\xonorm{\nabla& P_{\leq N}\big(F(u)\big)}{2}{6/5}
	\\ &\lesssim N\xonorm{F(u)-F(u_{\leq N/\eta^2})}{2}{6/5}+N^{1-s_c}\xonorm{\nsc F(u_{\leq N/\eta^2})}{2}{6/5}.
	\\ &\lesssim N\xonorm{u}{\infty}{3p/2}^p\xonorm{u_{>N/\eta^2}}{2}{6}
	+N^{1-s_c}\xonorm{u}{\infty}{3p/2}^p\xonorm{\nsc u_{\leq N/\eta^2}}{2}{6}
	\\ &\lesssim_u \eta N^{1-s_c}(1+N^{2s_c-1}K_I)^{1/2}.
	\end{align*}

This completes the proof of Proposition~\ref{prop:flmor2}.
\end{proof}


\subsection{The quasi-soliton scenario}\label{section:soliton2} In this section, we preclude the existence of solutions as in Theorem~\ref{thm:reduction1} with $3/4<s_c<1$ and
\begin{equation}\label{eq:infinite K2}
	K_{[0,T_{max})}=\int_0^{T_{max}}N(t)^{3-2s_c}\,dt=\infty.
\end{equation}

We rely on the frequency-localized Lin--Strauss Morawetz inequality established in Section~\ref{section:flmor2}. We also need the following lemma.

\begin{lemma}[Lower bound]\label{lem:mor lb2} Let $u:[0,T_{max})\times\R^3\to\C$ be an almost periodic solution as in Theorem~\ref{thm:reduction1} with $3/4<s_c<1$. Let $I\subset[0,T_{max}).$ Then there exists $N_0>0$ such that for any $N<N_0,$ we have
	\begin{equation}
	K_I\lesssim_u \iint_{I\times\R^3}\frac{\vert u_{>N}(t,x)\vert^{p+2}}{\vert x\vert}\,dx\,dt, 
	\label{eq:mor lb2}
	\end{equation}
where $K_I$ is as \eqref{def:K2}. 
\end{lemma}

\begin{proof}
We choose $C(u)$ and $N_0$ as in Lemma~\ref{lower bounds}. Then for $N<N_0$, we use H\"older and \eqref{lb3} to estimate
	\begin{align*}
	\iint_{I\times\R^3}\frac{\vert u_{> N}(t,x)\vert^{p+2}}{\vert x\vert}\,dx\,dt&\gtrsim_u \int_I N(t)\int_{\vert x\vert\leq\frac{C(u)}{N(t)}}\vert u_{> N}(t,x)\vert^{p+2}\,dx\,dt
	\\ &\gtrsim_u\int_I N(t)^{1+\frac{3p}{2}}\bigg(\int_{\vert x\vert\leq\frac{C(u)}{N(t)}}\vert u_{> N}(t,x)\vert^2\,dx\bigg)^{\frac{p+2}{2}}\,dt
	\\ &\gtrsim_u \int_I N(t)^{3+ps_c}\big(N(t)^{-2s_c}\big)^{\frac{p+2}{2}}\,dt
	\gtrsim_u K_I. 
	\end{align*}		\end{proof}

Finally, we prove the following.

\begin{theorem}[No quasi-solitons]\label{thm:soliton2} There are no almost periodic solutions as in Theorem~\ref{thm:reduction1} with $3/4<s_c<1$ such that \eqref{eq:infinite K2} holds.
\end{theorem}

\begin{proof} Suppose $u$ were such a solution. Let $\eta>0$ and let $I\subset[0,\infty)$ be a compact time interval, which is a contiguous union of characteristic subintervals.

Combining \eqref{eq:flmor2} and \eqref{eq:mor lb2}, we find that for $N$ sufficiently small, we have
		$$K_I\lesssim_u \eta(N^{1-2s_c}+K_I).$$
Choosing $\eta$ sufficiently small, we deduce $K_I\lesssim_u N^{1-2s_c}$ uniformly in $I$. We now contradict \eqref{eq:infinite K2} by taking $I$ sufficiently large inside of $[0,T_{max}).$ This completes the proof of Theorem~\ref{thm:soliton2}. 
\end{proof}


			\section{The case $0<s_c<1/2$}\label{section:case three}
In this section, we rule out the existence of solutions as in Theorem~\ref{thm:reduction2}, in which case $0<s_c<1/2$. As in Section~\ref{section:case two}, the main technical tool we use is a long-time Strichartz estimate, Proposition~\ref{prop:LTS}, which we prove in the next subsection. In Section~\ref{section:rfc}, we rule out the rapid frequency-cascade scenario; we show that such solutions are inconsistent with the conservation of energy. In Section~\ref{section:flmor}, we prove a frequency-localized Lin--Strauss Morawetz inequality, which we use in Section~\ref{section:soliton} to rule out the quasi-soliton scenario.

			\subsection{Long-time Strichartz estimates}\label{section:lts}
In this section, we establish a long-time Strichartz estimate adapted to the Lin--Strauss Morawetz inequality for almost periodic solutions as in Theorem~\ref{thm:reduction2}. A key ingredient in the proof is the bilinear Strichartz estimate, Proposition~\ref{lem:bilinear}. 

We work under the assumption
			\begin{equation}
			\label{eq:input}
			u\in L_t^\infty\dot{H}_x^{s}([0,\infty)\times\R^3)
			\end{equation}
for some $s\geq s_c$. We know from \eqref{a priori} that \eqref{eq:input} holds for $s=s_c$. In Section~\ref{section:rfc} we will show that rapid frequency-cascade solutions actually satisfy \eqref{eq:input} for $s>s_c$. 

	Throughout Section~\ref{section:case three}, we make use of the following notation:
			\begin{equation}
			\label{def:A}
			A_I(N):=\xnorm{u_{>N}}{2}{6}{I},
			\end{equation}
			
			\begin{equation}
			\label{def:K}
			K_I:=\int_I N(t)^{3-2s_c}\,dt\sim_u \sum_{J_k\subset I} N_k^{1-2s_c}.
			\end{equation}
	
	The main result of this section is the following.	
			
\begin{proposition}[Long-time Strichartz estimate]\label{prop:LTS}
Let $u:[0,\infty)\times\R^3\to\C$ be an almost periodic solution as in Theorem~\ref{thm:reduction2}. Let $I\subset[0,\infty)$ be a compact time interval, which is a contiguous union of characteristic subintervals $J_k$. Suppose \eqref{eq:input} holds for some $s_c\leq s<3/2+s_c$. Then for any $N>0$, we have
	\begin{equation}\label{eq:LTS}
	A_I(N)\lesssim_u N^{-s_c}+N^{-\sigma(s)}K_I^{1/2},
	\end{equation}
where $\sigma(s):=1/2+s-s_c$.
	
In particular, using \eqref{a priori}, we have
	\begin{equation}\label{eq:LTSsc}
	A_I(N)\lesssim_u N^{-s_c}+N^{-1/2}K_I^{1/2}.
	\end{equation}
Moreover, for any $\eps>0$, there exists $N_0(\eps)>0$ so that for $N\geq N_0$, 
	\begin{equation}\label{eq:LTS small}
	A_I(N)\lesssim_u \eps\big( N^{-s_c}+N^{-1/2}K_I^{1/2}\big).
	\end{equation} 
\end{proposition}			

We will prove Proposition~\ref{prop:LTS} by induction. The inductive step will rely on the following lemma.

\begin{lemma}\label{lem:recurrence} Let $\eta>0$ and  $u,$ $I,$ $s$, $\sigma$ be as above. For any $N>0$, we have
			\begin{align*}
			\xnorm{P_{>N}\big(F(u)\big)}{2}{6/5}{I}&\lesssim_u C_{\eta}\sup_{J_k\subset I}\norm{u_{>\eta N}}_{L_t^\infty\dot{H}_x^s(J_k\times\R^3)}^{2s_c}N^{-\sigma(s)}K_I^{1/2} 
		\\ &\quad\quad\quad+
			\sum_{M\leq \eta N}(\tfrac{M}{N})^{2}A_I(M),
			\end{align*}

\end{lemma}

\begin{proof}[Proof of Lemma~\ref{lem:recurrence}]
	Throughout the proof, all spacetime norms will be taken over $I\times\R^3$ unless stated otherwise.

	We begin by writing $F(u)=F(u_{\leq\eta N})+F(u)-F(u_{\leq\eta N}).$ We use Bernstein, the chain rule, Sobolev embedding, and \eqref{a priori} estimate
	\begin{align}
	\xonorm{P_{>N}F(u_{\leq\eta N})}{2}{6/5}
		&\lesssim N^{-2}\xonorm{\Delta F(u_{\leq\eta N})}{2}{6/5} \nonumber
		\\ &\lesssim N^{-2}\xonorm{\nsc u}{\infty}{2}^p\sum_{M\leq \eta N}\xonorm{\Delta u_M}{2}{6} \nonumber
		\\ &\lesssim_u \sum_{M\leq \eta N} (\tfrac{M}{N})^{2}A_I(M).\label{very low bound}
	\end{align}
	
	Next, we use almost periodicity to choose $C(\eta)$ large enough that
	\begin{equation}\label{Ceta}
	\xnorm{\nsc u_{>C(\eta)N(t)}}{\infty}{2}{I}<\eta^{2}.
	\end{equation}
By almost periodicity and the embedding $\dot{H}_x^{s_c}\hookrightarrow L_x^{3p/2}$ we may choose $C(\eta)$ possibly even larger to guarantee
	\begin{equation}\label{Ceta2}
	\xnorm{(1-\chi_{\frac{C(\eta)}{N(t)}})u_{\leq C(\eta)N(t)}}{\infty}{3p/2}{I}<\eta^{2},
	\end{equation}
where $\chi_R$ denotes the characteristic function of $\{\vert x\vert\leq R\}$. 

We now write $$F(u)-F(u_{\leq\eta N})\lesssim u_{>\eta N}\text{\O}\big\{(u_{\leq C(\eta)N(t)})^p+(u_{>C(\eta)N(t)})^p\big\},$$
so that 
	\begin{align}
	\xonorm{P_{>N}\big(F(u)-F(u_{\leq\eta N})\big)}{2}{6/5}
	&\lesssim \xonorm{u_{>\eta N}(u_{>C(\eta)N(t)})^p}{2}{6/5}\label{bootstrap1}
	\\ &+\xonorm{(1-\chi_{\frac{C(\eta)}{N(t)}})u_{>\eta N}(u_{\leq C(\eta)N(t)})^p}{2}{6/5}\label{bootstrap2}
	\\ &+\xonorm{\chi_{\frac{C(\eta)}{N(t)}}u_{>\eta N}(u_{\leq C(\eta)N(t)})^p}{2}{6/5}. \label{main contribution}
	\end{align} 
	
	Using H\"older, \eqref{a priori}, and \eqref{Ceta}, we estimate the contribution of \eqref{bootstrap1} as follows:	
	\begin{align}
	\xonorm{u_{>\eta N}(u_{>C(\eta)N(t)})^p}{2}{6/5}&
	\lesssim \xonorm{u_{>\eta N}}{2}{6}\xonorm{u_{>C(\eta)N(t)}}{\infty}{3p/2}^p\lesssim \eta^{2} A_I(\eta N).\label{bootstrap1 bound}
	\end{align}

	Similarly, we estimate the contribution of \eqref{bootstrap2} as follows:
	\begin{align}
	\xonorm{(1&-\chi_{\frac{C(\eta)}{N(t)}})u_{>\eta N}(u_{\leq C(\eta)N(t)})^p}{2}{6/5}\nonumber
	\\&\lesssim \xonorm{(1-\chi_{\frac{C(\eta)}{N(t)}})u_{\leq C(\eta)N(t)}}{\infty}{3p/2}\xonorm{u}{\infty}{3p/2}^{p-1}\xonorm{u_{>\eta N}}{2}{6}\nonumber
	\\&\lesssim_u \eta^{2} A_I(\eta N).\label{bootstrap2 bound}
	\end{align}
	
	Finally, we estimate the contribution of \eqref{main contribution}. We first restrict our attention to a single characteristic subinterval $J_k$. We define the following exponents:
$$q=\tfrac{2(p^2+2p-4)}{3p-4},\quad r_0=\tfrac{3p(p^2+2p-4)}{8p-p^2-8},\quad r=\tfrac{6(p^2+2p-4)}{3p^2-4}.$$ 
Note that as $4/3<p<2$, we have $4<q<\infty$, $2<r_0<6$, and $2<r<3$. We also note that we have the embedding $\dot{H}_x^{s_c,r}\hookrightarrow L_x^{r_0}$ and that $(q,r)$ is an admissible pair.
	
 With all spacetime norms over $J_k\times\R^3$, we use H\"older, the bilinear Strichartz estimate (Corollary~\ref{cor:bilinear}), Sobolev embedding, Lemma~\ref{lem:sb}, \eqref{a priori}, and \eqref{eq:input} to estimate
	\begin{align}
	&\xonorm{\chi_{\frac{C(\eta)}{N_k}}u_{>\eta N}(u_{\leq C(\eta)N_k})^p}{2}{6/5}\nonumber
	\\ &\lesssim \norm{\chi_{\frac{C(\eta)}{N_k}}}_{L_x^{\frac{6}{(1-2s_c)^2}}}\nonumber
		\xonorms{u_{>\eta N}u_{\leq C(\eta)N_k}}{2}^{1-2s_c}
		\xonorm{u_{>\eta N}}{\infty}{2}^{2s_c} \xonorm{u_{\leq C(\eta)N_k}}{q}{r_0}^{p-1+2s_c}
	\\ &\lesssim_u C_\eta N_k^{-(1-2s_c)^2/2}\big[N_k^{1-s_c}N^{-1/2-s}\big]^{1-2s_c}	\nonumber
	 N^{-2s\cdot s_c}\norm{u_{>\eta N}}_{L_t^\infty\dot{H}_x^s}^{2s_c}\xonorm{\nsc u}{q}{r}^{p-1+2s_c} \nonumber
	\\ &\lesssim_u C_\eta N_k^{1/2-s_c} N^{-(1/2+s-s_c)}\norm{u_{>\eta N}}_{L_t^\infty\dot{H}_x^s}^{2s_c}. \nonumber
	\end{align}
Summing over $J_k\subset I$ and using \eqref{def:K}, we find
	\begin{align}
	\xonorm{\chi_{\frac{C(\eta)}{N(t)}}u_{>\eta N}(u_{\leq C(\eta)N(t)})^p}{2}{6/5}\lesssim_u C_{\eta}\sup_{J_k\subset I}\norm{u_{>\eta N}}_{L_t^\infty\dot{H}_x^s(J_k\times\R^3)}^{2s_c}N^{-\sigma(s)}K_I^{1/2}.\label{main contribution bound}
	\end{align}

Adding the estimates \eqref{very low bound}, \eqref{bootstrap1 bound}, \eqref{bootstrap2 bound}, and \eqref{main contribution bound}, we complete the proof of Lemma~\ref{lem:recurrence}.
\end{proof}

We turn to the proof of Proposition~\ref{prop:LTS}.

\begin{proof}[Proof of Proposition~\ref{prop:LTS}]
We proceed by induction. For the base case, we let $N\leq \inf_{t\in I}N(t)\leq 1,$ so that $ N^{-2(s-s_c)}\big(\tfrac{N(t)}{N}\big)^{1-2s_c}\geq 1$ for $t\in I$. We use Bernstein and Lemma~\ref{lem:sb} to estimate
	\begin{align*}
	A_I(N)^2&\lesssim N^{-2s_c}\xnorm{\nsc u_{>N}}{2}{6}{I}^2
	\lesssim_u N^{-2s_c}+N^{-2s_c}\smallint_I N(t)^2\,dt
	\\ &\lesssim_u N^{-2s_c}+N^{-1-2(s-s_c)}K_I.
	\end{align*}
Thus for $N\leq \inf_{t\in I}N(t),$ we have
	\begin{equation}
	\label{base case}
	A_I(N)\leq C_u\big[N^{-s_c}+N^{-\sigma(s)}K_I^{1/2}\big].
	\end{equation}
This inequality remains true if we replace $C_u$ by any larger constant.
	
We now suppose that \eqref{base case} holds at frequencies $\leq N/2$; we will use Lemma~\ref{lem:recurrence} to show that it holds at frequency $N$. 

Applying Strichartz, Bernstein, Lemma~\ref{lem:recurrence}, \eqref{a priori}, and \eqref{eq:input},  we find
	\begin{align}
	A_I(N)\nonumber
	&\leq \tilde{C}_u\big[N^{-s_c}\inf_{t\in I}\norm{u_{>N}(t)}_{\dot{H}_x^{s_c}}+C_\eta \sup_{J_k\subset I}\norm{u_{>\eta N}}_{L_t^\infty\dot{H}_x^s(J_k\times\R^3)}^{2s_c} N^{-\sigma(s)}K_I^{1/2}\nonumber
	\\ &\quad\quad\quad+\sum_{M\leq\eta N}(\tfrac{M}{N})^{2}A_I(M)\big] \label{for smallness}
	\\ &\leq \tilde{C}_u\big[N^{-s_c}+C_\eta N^{-\sigma(s)}K_I^{1/2}+\sum_{M\leq \eta N}(\tfrac{M}{N})^{2}A_I(M)\big].\nonumber
	\end{align}
We now let $\eta<1/2$ and note that $s<3/2+s_c$ gives $\sigma(s)<2$. Thus, using the inductive hypothesis, we find
	\begin{align*}
	A_I(N)
	&\leq \tilde{C}_u\big[N^{-s_c}+C_\eta N^{-\sigma(s)} K_I^{1/2}+\sum_{M\leq\eta N}(\tfrac{M}{N})^{2}(C_uM^{-s_c}+C_uM^{-\sigma(s)}K_I^{1/2})\big]
	\\ &\leq \tilde{C}_u\big[N^{-s_c}+C_\eta N^{-\sigma(s)}K_I^{1/2}\big]+C_u\tilde{C}_u\big[\eta^{2-s_c} N^{-s_c}+\eta^{2-\sigma(s)}N^{-\sigma(s)}K_I^{1/2}\big].
	\end{align*}
If we now choose $\eta$ sufficiently small depending on $\tilde{C}_u$, we get
	$$A_I(N)\leq\tilde{C}_u(N^{-s_c}+C_\eta N^{-\sigma(s)}K_I^{1/2})+\tfrac12C_u(N^{-s_c}+N^{-\sigma(s)}K_I^{1/2}).$$ 
Finally, if we choose $C_u$ possibly larger so that $C_u\geq 2(1+C_\eta)\tilde{C}_u$, then the above inequality implies
	$$A_I(N)\leq C_u(N^{-s_c}+N^{-\sigma(s)}K_I^{1/2}),$$
as was needed to show. This completes the proof of \eqref{eq:LTS}. 

The estimate \eqref{eq:LTSsc} follows directly from \eqref{eq:LTS} with $s=s_c$. With \eqref{eq:LTSsc} in place, we can prove \eqref{eq:LTS small} by continuing from \eqref{for smallness}, choosing $\eta$ sufficiently small, and noting that $\sup_{t\in I}N(t)\leq 1$ implies
	$$\lim_{N\to\infty} \big[\inf_{t\in I}\norm{u_{>N}(t)}_{\dot{H}_x^{s_c}}+\sup_{J_k\subset I}\norm{u_{>\eta N}}_{L_t^\infty\dot{H}_x^{s_c}(J_k\times\R^3)}^{2s_c}\big]=0.$$ 
for any $\eta>0$. This completes the proof of Proposition~\ref{prop:LTS}.
\end{proof}


\subsection{The rapid frequency-cascade scenario}\label{section:rfc}

In this section, we preclude the existence of rapid frequency-cascade solutions, that is, almost periodic solutions $u$ as in Theorem~\ref{thm:reduction2} for which
		\begin{equation}\label{eq:finite K}
		K_{[0,\infty)}=\int_0^\infty N(t)^{3-2s_c}\,dt<\infty.
		\end{equation} 

We show that \eqref{eq:finite K} and Proposition~\ref{prop:LTS} imply that such a solution would possess additional regularity. We then use the additional regularity and the conservation of energy to derive a contradiction.

We note here that \eqref{eq:finite K} implies
		\begin{equation}
		\label{N to zero}
		\lim_{t\to\infty} N(t)=0.
		\end{equation}
		
We begin with the following lemma.

\begin{lemma}[Improved regularity]\label{lem:regularity} Let $u:[0,\infty)\times\R^3\to\C$ be an almost periodic solution as in Theorem~\ref{thm:reduction2}. Suppose 
	\begin{equation}
	\label{eq:input2}
	u\in L_t^\infty\dot{H}_x^{s}([0,\infty)\times\R^3)\quad\text{for some}\quad s_c\leq s<3/2+s_c.
	\end{equation}
If \eqref{eq:finite K} holds, then 
	\begin{equation}
	\label{eq:regularity}
	u\in L_t^\infty\dot{H}_x^{\sigma}([0,\infty)\times\R^3)\quad\text{for all}\quad s_c\leq \sigma<\sigma(s),
	\end{equation}	
where as above $\sigma(s):=1/2+s-s_c$.
\end{lemma}

\begin{proof} Throughout the proof, we take all spacetime norms over $[0,\infty)\times\R^3$. 

We will first use Proposition~\ref{prop:LTS} and \eqref{eq:finite K} to establish 
		\begin{equation}
		\label{improve A}
		A_{[0,\infty)}(N)\lesssim_u N^{-\sigma(s)}.
		\end{equation}
 
Let $I_n\subset[0,\infty)$ be a nested sequence of compact subintervals, each of which is a contiguous union of characteristic subintervals $J_k$. We let $\eta>0$ and apply Bernstein, Strichartz, Lemma~\ref{lem:recurrence}, and \eqref{eq:input2} to estimate
	\begin{align*}
	A_{I_n}(N)&\lesssim_u N^{-s_c}\inf_{t\in I_n}\norm{u_{>N}(t)}_{\dot{H}_x^{s_c}}+C_\eta N^{-\sigma(s)}K_{I_n}^{1/2}+\sum_{M\leq\eta N}(\tfrac{M}{N})^{2}A_{I_n}(M).
	\end{align*}
As \eqref{prop:LTS} gives $A_{I_n}(N)\lesssim_u N^{-s_c}+N^{-\sigma(s)}K_{I_n}^{1/2},$ we may choose $\eta$ sufficiently small and continue from above to get
	\begin{equation}\label{rfc lts}
	A_{I_n}(N)\lesssim_u N^{-s_c}\inf_{t\in I_n}\norm{u_{>N}(t)}_{\dot{H}_x^{s_c}}+N^{-\sigma(s)}K_{I_n}^{1/2}.
	\end{equation}

Using \eqref{N to zero}, we see that for any $N>0$ we have
	\begin{equation}\nonumber
	\lim_{t\to\infty}\norm{u_{>N}(t)}_{\dot{H}_x^{s_c}}=0.
	\end{equation}
Hence sending $n\to\infty$, continuing from \eqref{rfc lts}, and using \eqref{eq:finite K}, we get
	$$A_{[0,\infty)}(N)\lesssim_u N^{-\sigma(s)}.$$

We now show that \eqref{improve A} implies
	\begin{equation}
	\label{eq:mass decay}
	\xnorm{u_{>N}}{\infty}{2}{[0,\infty)}\lesssim_u N^{-\sigma(s)}.
	\end{equation}

We first use Proposition~\ref{reduced duhamel} and Strichartz to estimate
	$$\xonorm{u_{>N}}{\infty}{2}\lesssim \xonorm{P_{>N}\big(F(u)\big)}{2}{6/5}.$$

We write $F(u)=F(u_{\leq N})+F(u)-F(u_{\leq N})$. Noting that $s<3/2+s_c$ implies $\sigma(s)<2$, we use Bernstein, the chain rule, \eqref{a priori}, and \eqref{improve A} to estimate
	\begin{align*}
	\xonorm{P_{>N}\big(F(u_{\leq N})\big)}{2}{6/5}
	&\lesssim N^{-2}\xonorm{\nsc u}{\infty}{2}^p\sum_{M\leq N}\xonorm{\Delta u_M}{2}{6}
	\\ &\lesssim_u \sum_{M\leq N}(\tfrac{M}{N})^{2}M^{-\sigma(s)}
	\\ &\lesssim_u N^{-\sigma(s)}.
	\end{align*}

We next use H\"older, Sobolev embedding, \eqref{a priori}, and \eqref{improve A} to estimate
	$$
	\xonorm{P_{>N}\big(F(u)-F(u_{\leq N})\big)}{2}{6/5}
	\lesssim \xonorm{u}{\infty}{3p/2}^p\xonorm{u_{>N}}{2}{6}
	\lesssim_u N^{-\sigma(s)}.
	$$
Adding the last two estimates gives \eqref{eq:mass decay}.

Finally, we use \eqref{eq:mass decay} to prove \eqref{eq:regularity}. We fix $s_c\leq \sigma<\sigma(s)$ and use Bernstein, \eqref{a priori}, and \eqref{eq:mass decay} to estimate
	\begin{align*}
	\norm{\vert\nabla\vert^{\sigma}u}_{L_t^\infty L_x^2}&\lesssim\norm{\vert\nabla\vert^{s_c}u_{\leq 1}}_{L_t^\infty L_x^2}+\sum_{M>1}M^{\sigma}\norm{u_M}_{L_t^\infty L_x^2}
	\\ &\lesssim_u 1+\sum_{M>1} M^{\sigma-\sigma(s)}
	\lesssim_u 1. 
	\end{align*}
This completes the proof of Lemma~\ref{lem:regularity}. \end{proof}

We now iterate Lemma~\ref{lem:regularity} to establish additional regularity.

\begin{proposition}[Additional regularity]\label{prop:regularity} Let $u:[0,\infty)\times\R^3\to\C$ be an almost periodic solution as in Theorem~\ref{thm:reduction2}. If \eqref{eq:finite K} holds, then $u\in L_t^\infty\dot{H}_x^{1+\varepsilon}$ for some $\eps>0$.  
\end{proposition}

\begin{proof} As $0<s_c<1/2$, we may choose $c$ such that 
	$$2s_c<\tfrac{2}{3-2s_c}<c<\tfrac{3+2s_c}{4}<1.$$ 

We define $s_0=s_c$, and for $n\geq 0$ we define $s_{n+1}=c\cdot\sigma(s_{n}),$ where as above $\sigma(s):=1/2+s-s_c$. The constraint $c>2s_c$ guarantees that the sequence $s_n$ is increasing and bounded above by $\ell:=\tfrac{c(1-2s_c)}{2(1-c)}$. In fact, elementary arguments show that the sequence $s_n$ converges to $\ell$, and the constraint $c>\tfrac{2}{3-2s_c}$ guarantees that $\ell>1$. 

We have that $s_n\geq s_c$ for all $n\geq 0$, while the constraint $c<\tfrac{3+2s_c}{4}$ guarantees $s_n<3/2+s_c$ for all $n\geq 0$. Thus, noting that $s_n\leq s_{n+1}<\sigma(s_n)$ for each $n\geq 0$, we deduce from Lemma~\ref{lem:regularity} that
	$$u\in L_t^\infty\dot{H}_x^{s_n}\implies u\in L_t^\infty\dot{H}_x^{s_{n+1}}\quad\text{for all}\quad n\geq 0.$$ 
As \eqref{a priori} gives $u\in L_t^\infty\dot{H}_x^{s_0}$, we get by induction that $u\in L_t^\infty\dot{H}_x^{s_n}$ for all $n\geq 0$. As $s_n\to\ell>1$, we conclude that $u\in L_t^\infty\dot{H}_x^{1+\eps}$ for some $\eps>0$. 
\end{proof}

Combining Proposition~\ref{prop:regularity} with almost periodicity and the conservation of energy, we preclude the existence of rapid frequency cascades. 

\begin{theorem}[No frequency-cascades]\label{thm:rfc}
There are no almost periodic solutions $u$ as in Theorem~\ref{thm:reduction2} such that \eqref{eq:finite K} holds.
\end{theorem}
\begin{proof} Suppose $u$ were such a solution and let $\eta>0$. By almost periodicity, we may find $C(\eta)$ large enough that $\norm{\nsc u_{>C(\eta)N(t)}}_{L_t^\infty L_x^2}<\eta.$
Thus, by interpolation and Proposition~\ref{prop:regularity}, we have
	$$\norm{\nabla u_{>C(\eta)N(t)}}_{L_t^\infty L_x^2}
	\lesssim \norm{\nsc u_{>C(\eta)N(t)}}_{L_t^\infty L_x^2}^{\frac{\eps}{1+\eps-s_c}}
	\norm{\vert\nabla\vert^{1+\eps}u}_{L_t^\infty L_x^2}^{\frac{1-s_c}{1+\eps-s_c}}
	\lesssim_u \eta^{\frac{\eps}{1+\eps-s_c}}
	$$
for some $\eps>0$. 

On the other hand, by Bernstein and \eqref{a priori} we have
	$$\norm{\nabla u_{\leq C(\eta)N(t)}(t)}_{L_x^2}\lesssim_u [C(\eta)N(t)]^{1-s_c}\quad\text{for any}\quad t\in[0,\infty).$$ 

Thus we find
	\begin{equation}\nonumber\norm{\nabla u(t)}_{L_x^2}\lesssim_u \eta^{\frac{\eps}{1+\eps-s_c}}
	+[C(\eta)N(t)]^{1-s_c}\quad\text{for any}\quad t\in[0,\infty).\end{equation}
Using \eqref{N to zero} and the fact that $\eta>0$ was arbitrary, we deduce that	
	\begin{equation}\label{ke to zero}
	\norm{\nabla u(t)}_{L_x^2}\to 0\quad\text{as}\quad t\to\infty.
	\end{equation} 

We next use H\"older and Sobolev embedding to estimate
	\begin{align*}
	\norm{u(t)}_{L_x^{p+2}}\lesssim\norm{u(t)}_{L_x^{3p/2}}^{\frac{p}{p+2}}\norm{u(t)}_{L_x^{6}}^{\frac{2}{p+2}}\lesssim \norm{\nsc u(t)}_{L_x^2}^{\frac{p}{p+2}}\norm{\nabla u(t)}_{L_x^2}^{\frac{2}{p+2}},
	\end{align*}
so that \eqref{a priori} and \eqref{ke to zero} imply
	\begin{equation}\label{pe to zero}
	\norm{u(t)}_{L_x^{p+2}}\to 0\quad\text{as}\quad t\to\infty.
	\end{equation}
	
Adding \eqref{ke to zero} and \eqref{pe to zero} implies that $E[u(t)]\to 0$ as $t\to\infty.$ By the conservation of energy, we conclude $E[u(t)]\equiv 0$. Thus we must have $u\equiv 0$, which contradicts the fact that $u$ blows up. This completes the proof of Theorem~\ref{thm:rfc}.
\end{proof}


\subsection{A frequency-localized Lin--Strauss Morawetz inequality}\label{section:flmor}
	In this section, we use Proposition~\ref{prop:LTS} to prove a frequency-localized Lin--Strauss Morawetz inequality, which we will use in Section~\ref{section:soliton} to rule out the quasi-soliton scenario. As $s_c<1/2$, we prove an estimate that is localized to low frequencies.
	
	The main result of this section is the following.
\begin{proposition}[Frequency-localized Morawetz]\label{thm:flmor} Let $u:[0,\infty)\times\R^3\to\C$ be an almost periodic solution as in Theorem~\ref{thm:reduction2}. Let $I\subset[0,\infty)$ be a compact time interval, which is a contiguous union of characteristic subintervals $J_k$. Then for any $\eta>0$, there exists $N_0=N_0(\eta)$ such that for $N>N_0$, we have 
	\begin{equation}\label{eq:flmor}
	\iint_{I\times\R^3}\frac{\vert u_{\leq N}(t,x)\vert^{p+2}}{\vert x\vert}\,dx\,dt\lesssim_u \eta(N^{1-2s_c}+ K_I),
	\end{equation}
where $K_I$ is as in \eqref{def:K}. 
\end{proposition}

	To prove Proposition~\ref{thm:flmor}, we begin as in the proof of the standard Lin--Strauss Morawetz inequality \eqref{eq:ls}. We truncate the high frequencies of the solution and work with $u_{\leq N}$ for some $N>0$. As $u_{\leq N}$ is not a true solution to \eqref{nls}, we need to control error terms arising from this frequency projection. To do this, we choose $N$ large enough to capture `most' of the solution and use the estimates proved in Section~\ref{section:lts2}. We make these notions precise in the following lemma.

	\begin{lemma}[Low and high frequency control]\label{lem:control} Let $u,$ $I$, $K_I$ be as above. With all spacetime norms over $I\times\R^3$, we have the following.
	
	For any $N>0$ and $s>1/2$, 
	\begin{equation}\label{mor low}
	\xonorm{\vert\nabla\vert^s u_{\leq N}}{2}{6}\lesssim_u N^{s-s_c}(1+N^{2s_c-1}K_I)^{1/2}.
	\end{equation}

	For any $\eta>0$ and $s>s_c$, there exists $N_1=N_1(s,\eta)$ such that for $N>N_1$,
	\begin{equation}\label{low small}
	\xonorm{\vert\nabla\vert^{s} u_{\leq N}}{\infty}{2}\lesssim_u \eta N^{s-s_c}.
	\end{equation}
	
	For any $\eta>0$, there exists $N_2=N_2(\eta)>0$ such that for $N>N_2$, we have
	\begin{equation}\label{mor high}
	\xonorm{u_{>N}}{2}{6}\lesssim_u \eta N^{-s_c}(1+N^{2s_c-1}K_I)^{1/2}.
	\end{equation}

	\end{lemma}
	
	\begin{proof}
For \eqref{mor low}, we let $s>1/2$ and use \eqref{eq:LTSsc} to estimate
	\begin{align*}
	\xonorm{\vert\nabla\vert^{s}u_{\leq N}}{2}{6}&\lesssim\sum_{M\leq N}M^s\xonorm{u_M}{2}{6}
	\\ &\lesssim_u \sum_{M\leq N} M^{s-s_c}(1+M^{2s_c-1}K_I)^{1/2}
	\\ &\lesssim_u N^{s-s_c}(1+N^{2s_c-1}K_I)^{1/2}.
	\end{align*}
	
For \eqref{low small}, we first let $\eta>0$. Using almost periodicity and the fact that $\sup N(t)\leq 1$, we may find $C(\eta)>0$ so that $\xonorm{\nsc u_{>C(\eta)}}{\infty}{2}<\eta.$ Thus we can use Bernstein to see
	\begin{align*}
	\xonorm{\vert\nabla\vert^{s}& u_{\leq N}}{\infty}{2}
	\\ &\lesssim C(\eta)^{s-s_c}\xonorm{\nsc u_{\leq C(\eta)}}{\infty}{2}+N^{s-s_c}\xonorm{\nsc u_{C(\eta)<\cdot\leq N}}{\infty}{2}
	\\ &\lesssim_u C(\eta)^{s-s_c}+\eta N^{s-s_c}.
	\end{align*}
Choosing $N_1\gg \eta^{-1/(s-s_c)}C(\eta)$, we recover \eqref{low small}. 

Finally, we note that \eqref{mor high} is just a restatement of \eqref{eq:LTS small}. \end{proof}

We turn to the proof of Proposition~\ref{thm:flmor}.


\begin{proof}[Proof of Proposition~\ref{thm:flmor}] We take all spacetime norms over $I\times\R^3$.

We let $0<\eta\ll1 $ and choose 
	\begin{equation}\nonumber
	N>\max\{N_1(\tfrac12,\eta),N_1(\tfrac{1+s_c}{2},\eta^2),N_1(1,\eta),\tfrac{1}{\eta^2}N_2(\eta^2)\},
	\end{equation}
where $N_1$ and $N_2$ are as in Lemma~\ref{lem:control}. In particular, interpolating \eqref{mor low} and \eqref{low small} with $s=(1+s_c)/2$, we get
	\begin{equation}
	\label{43}
	\xonorm{\vert\nabla\vert^{(1+s_c)/2}u_{\leq N}}{4}{3}\lesssim_u\eta N^{(1-s_c)/2}(1+N^{2s_c-1}K_I)^{1/4}.
	\end{equation}
Moreover, as $\eta^2 N>N_2$, we can apply \eqref{mor high} to $u_{>\eta^2 N}$ to get
	\begin{equation}
	\label{high eta}
	\xonorm{u_{>\eta^2 N}}{2}{6}\lesssim \eta N^{-s_c}(1+N^{2s_c-1}K_I)^{1/2}.
	\end{equation}
 
We define the Morawetz action 
	$$\text{Mor}(t):=2\, \Im\int_{\R^3} \frac{x}{\vert x\vert}\cdot\nabla u_{\leq N}(t,x)\,\bar{u}_{\leq N}(t,x)\,dx.$$
A standard computation using $(i\partial_t+\Delta)u_{\leq N}=P_{\leq N}\big(F(u)\big)$ gives
	\begin{equation}\label{M'}
	\partial_t \text{Mor}(t)\gtrsim \int_{\R^3}\frac{x}{\vert x\vert}\cdot\{P_{\leq N}\big(F(u)\big),u_{\leq N}\}_P\,dx,\end{equation}
where the \emph{momentum bracket} $\{\cdot,\cdot\}_P$ is defined by $\{f,g\}_P:=\Re(f\nabla\bar{g}-g\nabla\bar{f}).$

Noting that $\{F(u),u\}_P=-\tfrac{p}{p+2}\nabla(\vert u\vert^{p+2}),$ we integrate by parts in \eqref{M'} to get
	\begin{align*}
	\partial_t \text{Mor}(t)&\gtrsim \int \frac{\vert u_{\leq N}(t,x)\vert^{p+2}}{\vert x\vert}\,dx
		+\int \frac{x}{\vert x\vert}\cdot\{P_{\leq N}\big(F(u)\big)-F(u_{\leq N}),u_{\leq N}\}_P\,dx.
	\end{align*}
Thus, by the fundamental theorem of calculus we have
	\begin{align}
	\iint_{I\times\R^3}&\frac{\vert u_{\leq N}(t,x)\vert^{p+2}}{\vert x\vert}\,dx\,dt \nonumber
	\\ &\lesssim \norm{\text{Mor}}_{L_t^\infty(I)}
	+\bigg\vert \iint_{I\times\R^3} \frac{x}{\vert x\vert}\cdot
		\{P_{\leq N}\big(F(u)\big)-F(u_{\leq N}),u_{\leq N}\}_P\,dx\,dt\bigg\vert. \nonumber
	\end{align}

To complete the proof of Proposition~\ref{thm:flmor}, it therefore suffices to show

\begin{equation}\label{lem:mor main}
\norm{\text{Mor}}_{L_t^\infty(I)}\lesssim_u \eta N^{1-2s_c},
\end{equation}
\begin{equation}\label{lem:mor error}
\bigg\vert \iint_{I\times\R^3} \frac{x}{\vert x\vert}\cdot
		\{P_{\leq N}\big(F(u)\big)-F(u_{\leq N}),u_{\leq N}\}_P\,dx\,dt\bigg\vert
		\lesssim_u \eta(N^{1-2s_c}+K_I).
\end{equation}

To prove \eqref{lem:mor main}, we use Bernstein, \eqref{hardys}, and \eqref{low small} to estimate
	\begin{align*}
	\norm{\text{Mor}}_{L_t^\infty(I)}&\lesssim \xonorm{\vert\nabla\vert^{-1/2}\nabla u_{\leq N}}{\infty}{2}
	\xonorm{\vert\nabla\vert^{1/2}(\tfrac{x}{\vert x\vert} u_{\leq N})}{\infty}{2}\nonumber
	\\ &\lesssim \xonorm{\vert\nabla\vert^{1/2}u_{\leq N}}{\infty}{2}^2
	\lesssim_u \eta N^{1-2s_c}.
	\end{align*}

We now turn to \eqref{lem:mor error}. We begin by rewriting
	\begin{align*}
	\{P_{\leq N}&\big(F(u)\big)-F(u_{\leq N}),u_{\leq N}\}_P
	\\ &=\{F(u)-F(u_{\leq N}),u_{\leq N}\}_P-\{P_{>N}\big(F(u)\big),u_{\leq N}\}_P
	\\ &=:I+II.
	\end{align*}

Writing
	$$I=\text{\O}\big\{[F(u)-F(u_{\leq N})]\nabla u_{\leq N}+u_{\leq N}\nabla[F(u)-F(u_{\leq N})]\big\}$$
and integrating by parts in the second term, we find that the contribution of $I$ to \eqref{lem:mor error} is controlled by
	\begin{align}
		&\xonorms{\nabla u_{\leq N}\big(F(u)-F(u_{\leq N})\big)}{1}					\label{Ia}
		\\ &\quad+\xonorms{\tfrac{1}{\vert x\vert}u_{\leq N}\big(F(u)-F(u_{\leq N})\big)}{1}.	\label{Ib}
	\end{align}

Similarly, writing 
	$$II=\text{\O}\big\{P_{>N}\big(F(u)\big)\nabla u_{\leq N}+\nabla P_{>N}\big(F(u)\big)u_{\leq N}\big\}$$
and integrating by parts in the second term, we find that the contribution of $II$ to \eqref{lem:mor error} is controlled by
	\begin{align}
		&\xonorms{\nabla u_{\leq N} P_{>N}\big(F(u)\big)}{1}						\label{IIa}
		\\&\quad+\xonorms{\tfrac{1}{\vert x\vert}u_{\leq N} P_{>N}\big(F(u)\big)}{1}		\label{IIb}
	\end{align}

To complete the proof of \eqref{lem:mor error}, it therefore suffices to show that the error terms \eqref{Ia} through \eqref{IIb} are acceptable, in the sense that they can be controlled by $\eta(N^{1-2s_c}+K_I).$ 

We first turn to \eqref{Ia}. Using H\"older, \eqref{a priori}, \eqref{mor low}, and \eqref{mor high}, we estimate
	\begin{align*}
	\eqref{Ia}&\lesssim \xonorm{\nabla u_{\leq N}}{2}{6}\xonorm{u_{>N}\text{\O}(u^p)}{2}{6/5}
	\\ &\lesssim \xonorm{\nabla u_{\leq N}}{2}{6}\xonorm{u_{>N}}{2}{6}\xonorm{u}{\infty}{3p/2}^p
	\\ &\lesssim_u\eta N^{1-2s_c}(1+N^{2s_c-1}K_I),
	\end{align*}
which is acceptable. 

We next turn to \eqref{Ib}. We first write
	\begin{align}
	\eqref{Ib}
	&\lesssim \xonorms{\tfrac{1}{\vert x\vert}u_{\leq N} (u_{>N})^{p+1}}{1}
	+\xonorms{\tfrac{1}{\vert x\vert}(u_{\leq N})^{p+1}u_{>N}}{1}\nonumber
	\end{align}

For the first piece, we use H\"older, Hardy, Bernstein, \eqref{a priori}, \eqref{low small}, and \eqref{mor high} to estimate
	\begin{align*}
	\xonorms{\tfrac{1}{\vert x\vert}u_{\leq N}(u_{>N})^{p+1}}{1}&
	\lesssim\xonorm{\tfrac{1}{\vert x\vert}u_{\leq N}}{\infty}{3p/2}\xonorm{u_{>N}}{2}{6}^2\xonorm{u}{\infty}{3p/2}^{p-1}
	\\&\lesssim_u\xonorm{\nabla u_{\leq N}}{\infty}{3p/2}\xonorm{u_{>N}}{2}{6}^2
	\\&\lesssim_uN^{s_c}\xonorm{\nabla u_{\leq N}}{\infty}{2}\xonorm{u_{>N}}{2}{6}^2
	\\&\lesssim_u \eta N^{1-2s_c}(1+N^{2s_c-1}K_I),
	\end{align*}
which is acceptable.

For the second piece, we use H\"older, Hardy, the chain rule, \eqref{a priori}, \eqref{mor low}, and \eqref{mor high} to estimate
	\begin{align*}
	\xonorms{\tfrac{1}{\vert x\vert}(u_{\leq N})^{p+1}u_{>N}}{1}
	&\lesssim\xonorm{\tfrac{1}{\vert x\vert}(u_{\leq N})^{p+1}}{2}{6/5}\xonorm{u_{>N}}{2}{6}
	\\ &\lesssim\xonorm{\nabla(u_{\leq N})^{p+1}}{2}{6/5}\xonorm{u_{>N}}{2}{6}
	\\ &\lesssim\xonorm{u}{\infty}{3p/2}^p\xonorm{\nabla u_{\leq N}}{2}{6}\xonorm{u_{>N}}{2}{6}
	\\ &\lesssim_u \eta N^{1-2s_c}(1+N^{2s_c-1}K_I),
	\end{align*}
which is acceptable. This completes the estimation of \eqref{Ib}.

We next turn to \eqref{IIa}. We first write
	$$\eqref{IIa}\lesssim\xonorms{\nabla u_{\leq N}P_{>N}\big(F(u_{\leq\eta^2 N})\big)}{1}+\xonorms{\nabla u_{\leq N}P_{>N}\big(F(u)-F(u_{\leq\eta^2 N})\big)}{1}.$$ 
For the first piece, we use H\"older, Bernstein, the chain rule, \eqref{a priori}, and \eqref{mor low} to estimate
	\begin{align*}
	\xonorms{\nabla u_{\leq N}P_{>N}\big(F(u_{\leq\eta^2 N})\big)}{1}
	&\lesssim N^{-1}\xonorm{\nabla u_{\leq N}}{2}{6}\xonorm{\nabla F(u_{\leq\eta^2 N})}{2}{6/5}
	\\ &\lesssim N^{-1}\xonorm{\nabla u_{\leq N}}{2}{6}\xonorm{u}{\infty}{3p/2}^p\xonorm{\nabla u_{\leq \eta^2 N}}{2}{6}
	\\ &\lesssim_u \eta N^{1-2s_c}(1+N^{2s_c-1}K_I),
	\end{align*}
which is acceptable.

For the second piece, we use H\"older, \eqref{a priori}, \eqref{mor low}, and \eqref{high eta} to estimate 
	\begin{align*}
	\xonorms{\nabla u_{\leq N}P_{>N}\big(F(u)-F(u_{\leq\eta^2 N})\big)}{1}&
	\lesssim \xonorm{\nabla u_{\leq N}}{2}{6}\xonorm{u_{>\eta^2 N}}{2}{6}\xonorm{u}{\infty}{3p/2}^p
	\\ &\lesssim_u \eta N^{1-2s_c}(1+N^{2s_c-1}K_I),
	\end{align*} 
which is acceptable. This completes the estimation of \eqref{IIa}. 

Finally, we turn to \eqref{IIb}. We first write
	$$\eqref{IIb}\lesssim\xonorms{\tfrac{1}{\vert x\vert}u_{\leq N} P_{>N}\big(F(u_{\leq \frac{N}{4}})\big)}{1}
		+\xonorms{\tfrac{1}{\vert x\vert}u_{\leq N}P_{>N}\big(F(u)-F(u_{\leq\frac{N}{4}})\big)}{1}.$$ 
For the first piece, we begin by noting that 
	$$P_{>N}\big(F(u_{\leq\frac{N}{4}})\big)=P_{>N}\big(P_{>\frac{N}{2}}(\vert u_{\leq \frac{N}{4}}\vert^p) u_{\leq\frac{N}{4}}\big).$$ 
Thus, using Cauchy--Schwarz, H\"older, Hardy, maximal function estimates, Bernstein, Sobolev embedding, \eqref{a priori}, \eqref{mor low}, and \eqref{43}, we can estimate
	\begin{align*}
	\xonorms{&\tfrac{1}{\vert x\vert}u_{\leq N}P_{>N}\big(F(u_{\leq\frac{N}{4}})\big)}{1}
	\\ &\lesssim \xonorms{\tfrac{1}{\vert x\vert}u_{\leq N}\M\big( P_{>\frac{N}{2}}(\vert u_{\leq \frac{N}{4}}\vert^p)u_{\leq\frac{N}{4}}\big)}{1}
	\\ &\lesssim \xonorms{\tfrac{1}{\vert x\vert}u_{\leq N}
	\big[\M\big(\big\vert P_{>\frac{N}{2}}(\vert u_{\leq\frac{N}{4}}\vert^p) \big\vert^2\big)\big]^{1/2}
	\big[\M\big(\vert u_{\leq \frac{N}{4}} \vert^2	\big)\big]^{1/2}}{1}
	\\ &\lesssim \xonorm{\tfrac{1}{\vert x\vert^{1/2}}u_{\leq N}}{4}{\frac{12p}{4+p}}
	\xonorm{\M\big(\big\vert P_{>\frac{N}{2}}(\vert u_{\leq\frac{N}{4}}\vert^p)\big\vert^2\big)}{1}{\frac{3p}{5p-4}}^{1/2}
	\xonorm{\tfrac{1}{\vert x\vert}\M\big(\vert u_{\leq\frac{N}{4}}\vert^2\big)}{2}{\frac{6p}{4+p}}^{1/2}
	\\ &\lesssim \xonorm{\vert\nabla\vert^{1/2}u_{\leq N}}{4}{\frac{12p}{4+p}}
	\xonorm{P_{>\frac{N}{2}}(\vert u_{\leq \frac{N}{4}}\vert^p)}{2}{\frac{6p}{5p-4}}
	\xonorm{\nabla\M(\vert u_{\leq \frac{N}{4}}\vert^2)}{2}{\frac{6p}{4+p}}^{1/2}
	\\ &\lesssim N^{-1}\xonorm{\vert\nabla\vert^{(1+s_c)/2}u_{\leq N}}{4}{3}
	\xonorm{\nabla(\vert u_{\leq\frac{N}{4}}\vert^p)}{2}{\frac{6p}{5p-4}}
	\xonorm{\nabla(\vert u_{\leq\frac{N}{4}}\vert^2)}{2}{\frac{6p}{4+p}}^{1/2}
	\\ &\lesssim N^{-1}\xonorm{\vert\nabla\vert^{(1+s_c)/2}u_{\leq N}}{4}{3}
	\xonorm{u}{\infty}{3p/2}^{p-1}\xonorm{\nabla u_{\leq\frac{N}{4}}}{2}{6}
		\xonorm{u}{\infty}{3p/2}^{1/2}\xonorm{\nabla u_{\leq \frac{N}{4}}}{2}{6}^{1/2}
	\\ &\lesssim_u N^{-1}\xonorm{\vert\nabla\vert^{(1+s_c)/2}u_{\leq N}}{4}{3}
		\xonorm{\nabla u_{\leq\frac{N}{4}}}{2}{6}
		\xonorm{\nabla u_{\leq\frac{N}{4}}}{2}{6}^{1/2}
	\\ &\lesssim_u \eta N^{1-2s_c}(1+N^{2s_c-1}K_I),
	\end{align*}	
which is acceptable. 
	
For the second piece, we write
	\begin{align}
	\xonorms{\tfrac{1}{\vert x\vert}u_{\leq N}P_{>N}\big(F(u)-F(u_{\leq\frac{N}{4}})\big)}{1}
	&\lesssim
	\xonorms{\tfrac{1}{\vert x\vert}u_{\leq N}\M\big(u_{>\frac{N}{4}}(u_{\leq \frac{N}{4}})^p\big)}{1}	\label{final2}
	\\ &\quad+\xonorms{\tfrac{1}{\vert x\vert}u_{\leq N}\M\big((u_{>\frac{N}{4}})^{p+1}\big)}{1}.		\label{final1}
	\end{align}

For \eqref{final2}, we use Cauchy--Schwarz, H\"older, Hardy, the maximal function estimate, Sobolev embedding, \eqref{a priori}, \eqref{mor low}, \eqref{mor high}, and \eqref{43} to estimate
	\begin{align*}
	&\eqref{final2}
	\\ &\lesssim\xonorms{\tfrac{1}{\vert x\vert}u_{\leq N}\big[\M\big(\vert u_{>{\frac{N}{4}}}\vert^2\vert u_{\leq\frac{N}{4}}\vert^{2(p-1)}\big)\big]^{1/2}\big[\M\big(\vert u_{\leq\frac{N}{4}}\vert^2\big)\big]^{1/2}}{1}
	\\ &\lesssim\xonorm{\tfrac{1}{\vert x\vert^{1/2}}u_{\leq N}}{4}{\frac{12p}{4+p}}\xonorm{\M\big(\vert u_{>\frac{N}{4}}\vert^2\vert u_{\leq\frac{N}{4}}\vert^{2(p-1)}\big)}{1}{\frac{3p}{5p-4}}^{1/2}
	\xonorm{\tfrac{1}{\vert x\vert}\M\big(\vert u_{\leq\frac{N}{4}}\vert^2)}{2}{\frac{6p}{4+p}}^{1/2}
	\\ &\lesssim\xonorm{\vert\nabla\vert^{1/2}u_{\leq N}}{4}{\frac{12p}{4+p}}
		\xonorm{u_{>\frac{N}{4}}}{2}{6}
		\xonorm{u}{\infty}{3p/2}^{p-1}
		\xonorm{\nabla\M\big(\vert u_{\leq\frac{N}{4}}\vert^2\big)}{2}{\frac{6p}{4+p}}^{1/2}
	\\ &\lesssim\xonorm{\vert\nabla\vert^{(1+s_c)/2}u_{\leq N}}{4}{3}
		\xonorm{u_{>\frac{N}{4}}}{2}{6}\xonorm{u}{\infty}{3p/2}^{p-1/2}
		\xonorm{\nabla u_{\leq\frac{N}{4}}}{2}{6}^{1/2}
	\\ &\lesssim_u \eta N^{1-2s_c}(1+N^{2s_c-1}K_I),
	\end{align*}
which is acceptable.

Finally, we use H\"older, Hardy, Bernstein, \eqref{a priori}, \eqref{low small}, and  \eqref{mor high} to estimate
	\begin{align*}
	\eqref{final1}&\lesssim \xonorm{\tfrac{1}{\vert x\vert}u_{\leq N}}{\infty}{3p/2}\xonorm{u_{>\frac{N}{4}}}{2}{6}^2\xonorm{u}{\infty}{3p/2}^{p-1}
	\\ &\lesssim_u \xonorm{\nabla u_{\leq N}}{\infty}{3p/2}\xonorm{u_{>\frac{N}{4}}}{2}{6}^2
	\\ &\lesssim_u N^{s_c}\xonorm{\nabla u_{\leq N}}{\infty}{2}\xonorm{u_{>\frac{N}{4}}}{2}{6}^2
	\\ &\lesssim_u \eta N^{1-2s_c}(1+N^{2s_c-1}K_I),
	\end{align*}
which is acceptable. This completes the estimation of the final error term \eqref{IIb}, which in turn completes the proof of Proposition~\ref{thm:flmor}. \end{proof}


\subsection{The quasi-soliton scenario}\label{section:soliton}
	In this section, we preclude the existence of solutions as in Theorem~\ref{thm:reduction2} for which
	\begin{equation}
	\label{eq:infinite K}
	K_{[0,\infty)}=\int_0^\infty N(t)^{3-2s_c}\,dt=\infty.
	\end{equation}
	
	We rely primarily on the frequency-localized Lin--Strauss Morawetz established in Section~\ref{section:flmor}. We also need the following lemma.
	
	\begin{lemma}[Lower bound]\label{lem:mor lb}
	Let $u:[0,\infty)\times\R^3\to\C$ be an almost periodic solution as in Theorem~\ref{thm:reduction2}. Let $I\subset[0,\infty)$. Then there exists $N_1>0$ such that for any $N>N_1$, we have
	\begin{equation}\label{eq:mor lb}
	K_I\lesssim_u \iint_{I\times\R^3}\frac{\vert u_{\leq N}(t,x)\vert^{p+2}}{\vert x\vert}\,dx\,dt,
	\end{equation}
where $K_I$ is as in \eqref{def:K}. 
	\end{lemma}
	
	\begin{proof}
We choose $C(u)$ and $N_1$ as in Lemma~\ref{lower bounds}. Then for $N>N_1$, we use H\"older and \eqref{lb2} to estimate
	\begin{align*}
	\iint_{I\times\R^3}\frac{\vert u_{\leq N}(t,x)\vert^{p+2}}{\vert x\vert}\,dx\,dt&\gtrsim_u \int_I N(t)\int_{\vert x\vert\leq\frac{C(u)}{N(t)}}\vert u_{\leq N}(t,x)\vert^{p+2}\,dx\,dt
	\\ &\gtrsim_u\int_I N(t)^{1+\frac{3p}{2}}\bigg(\int_{\vert x\vert\leq\frac{C(u)}{N(t)}}\vert u_{\leq N}(t,x)\vert^2\,dx\bigg)^{\frac{p+2}{2}}\,dt
	\\ &\gtrsim_u \int_I N(t)^{3+ps_c}\big(N(t)^{-2s_c}\big)^{\frac{p+2}{2}}\,dt
	\gtrsim_u K_I. 
	\end{align*}\end{proof}

Finally, we prove the following.

\begin{theorem}[No quasi-solitons]\label{thm:soliton} There are no almost periodic solutions as in Theorem~\ref{thm:reduction2} such that \eqref{eq:infinite K} holds.  
\end{theorem}

\begin{proof}
Suppose $u$ were such a solution. Let $\eta>0$ and let $I\subset[0,\infty)$ be a compact time interval, which is a contiguous union of characteristic subintervals. 

Combining \eqref{eq:flmor} and \eqref{eq:mor lb}, we find that for $N$ sufficiently large, we have
	$$K_I\lesssim_u \eta(N^{1-2s_c}+K_I).$$
Choosing $\eta$ sufficiently small, we deduce $K_I\lesssim_u N^{1-2s_c}$ uniformly in $I$. We now contradict \eqref{eq:infinite K} by taking $I$ sufficiently large inside of $[0,\infty)$. This completes the proof of Theorem~\ref{thm:soliton}.
\end{proof}

\end{document}